\documentclass[final]{siamltex}

\usepackage{amsmath}
\usepackage{amscd}
\usepackage{amssymb}
\usepackage{epsfig}

\numberwithin{equation}{section}

\newcommand{\la}{\lambda}

\newcommand{\rz}{\mathbb{R}}

\title{Quadratic Interpolation and Rayleigh-Ritz Methods for Bifurcation Coefficients\thanks{
Received by the editors February 14, 2009; accepted for publication (in revised form) XXXX XX, XXXX; published electronically XXXX XX, XXXX.
}}

\author{W. M. Greenlee
\thanks{Department of Mathematics,
University of Arizona,
617 Santa Rita, Tucson, AZ 85721 USA
({\tt mgrnle@math.arizona.edu}).}
\and L. Hermi
\thanks{Department of Mathematics,
University of Arizona,
617 Santa Rita, Tucson, AZ 85721 USA
({\tt hermi@math.arizona.edu}).}
}

\begin{document}

\maketitle

%\footnotetext[1]{ Partially supported by National Science
%Foundation (USA) grant DMS--9870156. } \footnotetext[2]{
%Corresponding author. }

\medskip

{\bf Abstract.} In this article we study the estimation of bifurcation coefficients in nonlinear branching problems by means of Rayleigh-Ritz approximation to the eigenvectors of the corresponding linearized problem. It is essential that the approximations converge in a norm of sufficient strength to render the nonlinearities continuous. Quadratic interpolation between Hilbert spaces is used to seek sharp rate of convergence results for bifurcation coefficients. Examples from ordinary and partial differential problems are presented.

\medskip


\medskip


%\tableofcontents

\newpage

\thispagestyle{plain}
\markboth{W. M. Greenlee AND L. Hermi}{Fractional and Harmonic Rayleigh-Ritz}

\section{Introduction} \label{intro}

The most common method for estimation of the lower eigenvalues of a differential operator is the Rayleigh-Ritz method, or a variant such as finite elements in engineering problems, or Hartree-Fock in quantum mechanical problems. The emphasis in many numerical studies is on obtaining accurate eigenvalue approximations in an efficient and cost effective manner. Herein, we explore a different question, namely, is the calculation useful if needed to estimate a bifurcation coefficient in a nonlinear branching problem? If the eigenfunctions of the corresponding linearized problem must be approximated, the strength of the nonlinearity must be considered. In particular, the eigenfunctions of the linearized problem need to be approximated in a topology strong enough that nonlinear quantities to be calculated behave continuously. This may or may not be the case with the standard Rayleigh-Ritz method. We know of one prior paper on this question \cite{Bazley-Zwahlen}, in which the nonlinearity is taken to be continuous in the underlying Hilbert space topology.

Our approach is to introduce a whole scale of Rayleigh quotients using real powers of the selfadjoint operator which determines the corresponding linearized problem. This is a version of quadratic interpolation, also called Banach space interpolation between Hilbert spaces. A discrete set of such Rayleigh quotients was called ``Schwartz quotients'' in \cite{Collatz}, and used to generate a theoretical algorithm for eigenvalue estimation that is closely related to the power method for matrices.  We use our somewhat esoteric construction to seek sharp convergence rate estimates for bifurcation coefficients. This proceeds by use of eigenfunction estimates in Hilbert space topologies strong enough for the nonlinear approximation procedure to converge.

In the next section we describe the basic Rayleigh-Ritz method, and formulate a model problem concerning a nonlinear rotating string with variable density. In Section 3 the scale of Rayleigh quotients is presented, and relevant facts about quadratic interpolation between closed subspaces of Sobolev spaces determined by homogeneous boundary conditions are stated. The basic eigenvalue-eigenvector convergence theorem is also presented here, with the proof attached as the appendix. Then in Section 4 we prove theorems on spectral implementation of the basic convergence theorem. This method is known as the spectral Galerkin method in the numerical analysis literature (cf. \cite{Bernardi-Maday}, \cite{Boyd}, \cite{Mercier}). Most of this section can be implemented via finite elements, which we comment on in Section 6. In Section 5 we analyze the nonlinear rotating string problem, and present a few numerical results for this problem with a particular density function. Finally in Section 6 we illustrate our theory in the context of bifurcation problems for some partial differential equations, and close with a few remarks.

\section{The Rayleigh-Ritz Method and a Model Nonlinear Problem} \label{rrm}

Let $\mathcal{H}$ be a separable complex Hilbert space with norm
$\| \cdot \|$ and inner product $\langle \cdot, \cdot \rangle$.
Let $A$ be a selfadjoint operator in  $\mathcal{H}$ with domain
$\mathcal{D}(A)$. We assume that $A$ is positive definite, i.e.,
$\|A u \| \ge \alpha \|u\|$, $u \in \mathcal{D}(A)$, $\alpha >0$,
and that the lower portion of the spectrum of $A$ consists of
isolated eigenvalues of finite multiplicity. Thus the lower
spectrum of $A$ may be written
\[0<\alpha \le \la_1 \le \la_2 \le \ldots \le \la_{\infty}\]
with corresponding orthonormal eigenvectors $u_1, u_2, \ldots,$
where $\la_{\infty}$ denotes the least point of the essential
spectrum of $A$. By convention, $\la_{\infty} = \infty$ if $A$ has
compact inverse. Further, let $a(u)$ be the quadratic form
corresponding to $A$, i.e., the closure of $\langle A u , u
\rangle$, with domain $\mathcal{D}(a)$. Then $a(u) \ge \alpha \|u
\|^2$ for $u \in \mathcal{D}(a)$, and we let $a(u,v)$, $u, v \in
\mathcal{D}(a)$, be the corresponding Hermitian symmetric bilinear
form.

The Rayleigh-Ritz method for estimation of eigenpairs of $A$ begins with a family of trial vectors $\{w_i\}_{i=1}^{n} \subset \mathcal{D}(a)$. It then proceeds with the setup and resolution of the generalized matrix eigenvalue problem
\begin{equation} \notag
\big[a(w_i, w_j) \big] \, {\bf x} =\Lambda \big[\langle w_i, w_j \rangle\big] \, {\bf x}.
\end{equation}
If the resulting matrix eigenvalues are denoted by $\Lambda_1^{(n)} \le \Lambda_2^{(n)} \le \ldots \le \Lambda_n^{(n)}$, the first monotonicity principle \cite{Weinberger-SIAM} guarantees that for each $i=1, 2, \ldots, n$, $\lambda_i \le \Lambda_i^{(n)}$, thus providing upper bounds for the lowest $n$ eigenvalues of $A$. Let $u_1^{(n)}, u_2^{(n)}, \ldots, u_n^{(n)}$ be the corresponding orthogonalized eigenvectors, normalized in $\mathcal{H}$. Thus $\langle u_i^{(n)}, u_j^{(n)}\rangle= \delta_{ij}$ and
$a(u_i^{(n)}, u_j^{(n)}) = \Lambda_i^{(n)} \, \delta_{ij}$, where $\delta_{ij}$ is the Kronecker symbol.

Completeness of $\{w_i\}_{i=1}^{\infty}$ in $\mathcal{D}(a)$ is sufficient to guarantee that for each eigenvalue $\lambda_i$ of $A$ below $\lambda_{\infty}$, $\Lambda_i^{(n)} \to \lambda_i$ and $u_i^{(n)} \to u_i$ in $\mathcal{D}(a)$ both as $n\to \infty$. (cf. \cite{Babuska-Osborn} \cite{Chatelin} \cite{Weinberger-SIAM}). Rate of convergence estimates are also available from these sources. But to use the Rayleigh-Ritz eigenvectors for other calculations of interest requires estimates in other norms, which is the thrust of this paper. Herein we look specifically at the estimation of bifurcation coefficients in nonlinear eigenvalue problems. Convergence in the energy norm, i.e., that of $\mathcal{D}(a)$, may be insufficient to handle the nonlinearity, or when sufficient may give only a crude estimate. To illustrate this point we now begin the examination of a nonlinear rotating string problem.

Perhaps the simplest nonlinear refinement of the standard linear model for a tightly stretched flexible string of length $L$ with linear density $\rho$, rotating with uniform angular velocity $\omega$ about its equilibrium position along the $x-$axis is
\begin{equation} \label{5.1}
T_0 \, \frac{d^2 y}{dx^2} + \rho \, \omega^2 \, y \, \sqrt{1+\left( \frac{dy}{dx}\right)^2}=0.
\end{equation}
Herein, $y$ is the deflection from equilibrium at the point $x$, and the condition of force equilibrium in the $x$ direction (in the absence of external forces in that direction) implies that the tension $T$ satisfies
\[T \, \frac{dx}{ds}= T_0=\text{ constant},\]
so that by solving for the deflection $y$, the tension $T$ is given by
\[T=T_0 \, \sqrt{1+\left( \frac{dy}{dx}\right)^2}\]
(cf. \cite{Hilde}). As usual, $s$ denotes arclength. When $\rho$ is constant, \eqref{5.1} is solvable explicitly in terms of elliptic integrals \cite{Greenhill}. In this section we examine calculation of small, but not infinitesimal, deflections when $\rho$ is variable. Setting $\lambda=\omega^2/T_0$ and $y=\sqrt{\epsilon} \, u$, with $\epsilon$ small and positive, yields the nonlinear eigenvalue problem
\begin{equation}
-u'' = \lambda \, \rho(x) u \, \sqrt{1+ \epsilon \, u'^2}=0, \quad ':=\frac{d}{dx}\notag
\end{equation}
and we will require fixed end conditions
\[u(a)=u(b)=0,\]
signifying that the equilibrium position of the string is the interval $[a,b]$ of length $L$. We assume that $\rho$ is smooth and strictly positive on $[a,b]$.

To better illustrate applicability of the theory we will in fact consider the family of nonlinear eigenvalue problems
\begin{equation} \label{5.2}
-u'' = \lambda \, \rho(x) u \, \left(1+ \epsilon \, \left(u'\right)^{2 t}\right)^{\frac{1}{2t}}=0, \quad
u(a)=u(b)=0,
\end{equation}
where $t$ is a positive integer. These are all $W^{1,1}(a,b)$ perturbations of the linearized problem obtained by setting $\epsilon=0$, and the rotating string model is given when $t=1$.

Solutions of \eqref{5.2} are smooth and the eigenvalues of the linearized problem are simple, so known results of bifurcation theory employing the Banach space $C^2[a,b]$ imply that there are solution pairs $(\lambda_{\epsilon}, u_{\epsilon})$ analytic in $\epsilon$ for small $\epsilon$
emanating from any eigenpair $(\lambda_0, u_0)$ of
\begin{equation} \label{5.3}
-u'' = \lambda \, \rho(x) u, \quad
u(a)=u(b)=0,
\end{equation}
(cf. \cite{Sattinger}, \cite{Vainberg-Trenogin}). So we write solutions of \eqref{5.2} as
\begin{align} \label{5.4}
u_{\epsilon}&= v_0+ \epsilon \, v_1 + \frac{\epsilon^2}{2} \, v_2+\ldots \notag \\
& \\
\lambda_{\epsilon}&= \nu_0+ \epsilon \, \nu_1 + \frac{\epsilon^2}{2} \, \nu_2+\ldots, \notag
\end{align}
where $v_0=u|_{\epsilon=0}$, $v_1=\frac{d u}{d \epsilon}|_{\epsilon=0}$, etc. As we shall see, calculation of the coefficients $v_1, \nu_1$, etc. in \eqref{5.4} based on Rayleigh-Ritz type solutions of \eqref{5.3} will require Ritz methods convergent in norms stronger than that of the energy norm, i.e. the norm of $H_0^1(a,b)$. That this is the case is readily seen by expanding the nonlinearity in \eqref{5.2} to obtain
\[-u''= \lambda \, \rho \, u \left(1+ \frac{\epsilon}{2 t} \, \left(u'\right)^{2 t} + \frac{(1-2t)}{8 t^2} \epsilon^2 \, \left(u'\right)^{4 t}+\ldots \right).\]
Successive equations for $\nu_0, v_0, \nu_1, v_1, \ldots$ are obtained, as usual, by equating coefficients in $\epsilon$ in this expression.
%As noted in \cite{Keller}, this is somehow easier to do by formally %differentiating both sides with respect to $\epsilon$ and then setting %$\epsilon=0$, than by substituting the power series for $\lambda$ and $u$ in %\eqref{5.4} and equating coefficients.
The first equation obtained is the linear eigenvalue problem
\begin{equation} \label{5.5}
-v_0''=\nu_0 \, \rho(x) \, v_0, \quad v_0(a)=v_0(b)=0.
\end{equation}
Our goal is to investigate the use of Rayleigh-Ritz methods for \eqref{5.5} in the calculation of the bifurcation coefficient $\nu_1$, and further terms in the two series
\eqref{5.4}. The second equation is
\begin{equation} \label{5.6}
-v_1''-\nu_0\, \rho(x) v_1=\nu_1 \, \rho(x) \, v_0 +\frac{\nu_0}{2 t} \, \rho(x) \, v_0 \, \left(v_0'\right)^{2 t}, \quad v_1(a)=v_1(b)=0.
\end{equation}
The Fredholm alternative yields solvability of \eqref{5.6} if and only if
\[\nu_1 \, \left(\rho(x) v_0, v_0\right) + \frac{\nu_0}{2 t} \left(\rho(x) v_0 (v_0')^{2t}, v_0\right)=0,\]
where $\left(\cdot\, , \cdot\right)$ is the usual $L^2(a,b)$ inner product with respect to Lebesgue measure $dx$. Taking $v_0$ to be normalized in $L^2$ with respect to the measure $\rho(x) dx$ over $(a,b)$ (which we will take as the space $\mathcal{H}$ with inner product $\langle \cdot, \cdot \rangle$), i.e.,
\begin{equation} \label{5.7}
\left(\rho(x) v_0, v_0\right) = \langle v_0, v_0\rangle=1,
\end{equation}
gives $\nu_1$ as
\begin{equation} \label{5.8}
\nu_1= -\frac{\nu_0}{2 t} \, \left(\rho(x) v_0 (v_0')^{2t}, v_0\right)=  -\frac{\nu_0}{2 t} \, \langle
v_0 (v_0')^{2t}, v_0 \rangle.
\end{equation}
The component of $v_1$ in the ``$\rho$ orthogonal complement of $v_0$'' is uniquely determined by solving \eqref{5.6} with this value of $\nu_1$. Note that $\nu_1$ is negative, i.e., the bifurcation is subcritical. By also setting the component of $v_1$ parallel to $v_0$ in $\mathcal{H}$ equal to zero, one obtains that the parameter $\epsilon^{\frac{1}{2t}}$ is the $\mathcal{H}$ norm of the component of the power series solution of \eqref{5.1} in the direction of $v_0$.

The third equation obtained from \eqref{5.4} is
\begin{align} \label{5.9}
-v_2'' - \nu_0 \, \rho(x) v_2 &= \nu_2 \, \rho(x) v_0 + 2 \nu_1 \rho(x) v_1 + \frac{\nu_1}{t} \, \rho(x) \, v_0 \left(v_0'\right)^{2t} \notag \\
&+ \frac{\nu_0}{t} \,  \rho(x) \, v_1 \, \left(v_0'\right)^{2t}+ 2 \nu_0 \, \rho(x) \, v_0 \, v_1' \left(v_0'\right)^{2t-1}\notag \\
&+\frac{\nu_0 (1-2t)}{4 t^2} \, \rho(x) \, v_0 \, \left(v_0'\right)^{4t},
\end{align}
with $v_2 (a) =v_2(b)=0$. Retaining the normalization \eqref{5.7}, the Fredholm
alternative yields solvability of \eqref{5.9} if and only if
\begin{align} \label{5.10}
\nu_2 &= - 2 \nu_1 \langle v_1, v_0\rangle - \frac{\nu_1}{t} \, \langle v_0 \left(v_0'\right)^{2t}, v_0\rangle \notag \\
&- \frac{\nu_0}{t} \,  \langle v_1 \, \left(v_0'\right)^{2t}, v_0\rangle
- 2 \nu_0 \, \langle v_0 \, v_1' \left(v_0'\right)^{2t-1}, v_0\rangle \notag \\
&-\frac{\nu_0 (1-2t)}{4 t^2} \, \langle v_0 \, \left(v_0'\right)^{4t}, v_0\rangle.
\end{align}
With this value of $\nu_2$, $v_2$ is uniquely determined in the $\mathcal{H}$ orthogonal complement of $v_0$, while in the direction of $v_0$ the component of $v_2$ is taken to be zero. This preserves the previous interpretation of the parameter $\epsilon^{\frac{1}{2t}}$.

Note that if \eqref{5.5} is solved approximately by the Rayleigh-Ritz method, convergence of the approximate eigenvectors is in the topology of $H_0^1(a,b)$. This is sufficient to guarantee convergence of the corresponding approximation to $\nu_1$ as calculated from \eqref{5.8} only for $t=1$. To rigorously approximate $\nu_1$ for $t>1$, or $\nu_2$ for any value of $t$, requires that approximate eigenvectors converge in $C_0^1[a,b]$, which is strictly contained in $H_0^1(a,b)$. We will obtain such convergence estimates by using quadratic interpolation between Hilbert spaces to generalize the venerable notion of Rayleigh quotient. This will yield convergence rates in $H^{\alpha}(a,b)$ with $\alpha >3/2$, which implies convergence in $C_0^1[a,b]$ via basic Sobolev theory. Then rates of convergence of these nonlinear functionals will be estimated.

\section{Quadratic Interpolation and Fractional Rayleigh-Ritz} \label{quadratic-interpolation}

To provide a theoretical framework for the eigenvector estimates required in the preceding example, let $T=A^{-1}$, which is a bounded selfadjoint operator in $\mathcal{H}$. It is well known that $T$ is also selfadjoint in $\mathcal{D}(a)$ with the topology induced by the inner product $a(u,v)$ i.e., the restriction of $T$ to $\mathcal{D}(a)$ is bounded and satisfies $a(T u, v) = a(u, T v) = \langle u, v\rangle$, $u, v, \in \mathcal{D}(a)$. The Rayleigh quotient for reciprocals of eigenvalues of $A$ is
\begin{equation}
\frac{\langle u, u\rangle}{a(u)} = \frac{a(T u, u)}{a(u)}.
\notag
\end{equation}
In this form the Rayleigh-Ritz method produces lower bounds, \[\left(\Lambda_i^{(n)}\right)^{-1} \nearrow \lambda_i^{-1}\] as $n \to \infty$ from trial vectors $\{ w_i \}_{i=1}^{\infty}$ that are complete in $\mathcal{D}(a)$. Now for real $\tau$ let $A^{\tau}$ be the $\tau^{\text{th}}$ power of $A$ as defined by use of the spectral theorem \cite{Kato}. This is particularly simple to describe if $A$ has a complete set of eigenvectors $\{ u_i \}_{i=1}^{\infty}$, orthonormal in $\mathcal{H}$. Then $u=\sum_{i=1}^{\infty} \alpha_i \, u_i \in \mathcal{H}$ if and only if $\sum_{i=1}^{\infty} |\alpha_i|^2 <\infty$, and if $\tau>0$, $u \in \mathcal{D}(A^{\tau})$ if and only if $\sum_{i=1}^{\infty} \lambda_i^{2 \tau} \, |\alpha_i|^2 <\infty$, while for $\tau<0$ we take $\mathcal{D}(A^{\tau})$ to be the completion of $\mathcal{H}$ in the norm $\|A^{\tau} u \|^2 = \sum_{i=1}^{\infty} \lambda_i^{2 \tau} \, |\alpha_i|^2$. Herein $\alpha_i =\langle u, u_i \rangle$. Whether or not $A$ has a complete set of eigenvectors, $\mathcal{D}(A^{\tau})$
and $\mathcal{D}(A^{-\tau})$ are dual spaces, where duality is taken relative to $\langle \cdot, \cdot \rangle$. Since $A$ has a strictly positive lower bound, if $A$ is bounded all of these norms generate the same topology. Our interest is in unbounded $A$, in particular differential operators, and then if $\tau_1>\tau_2$, the topology corresponding to $\tau_1$ is strictly stronger than that for $\tau_2$. When $A$ is differential and $\tau$ is positive and not an integer, $A^{\tau}$ acts as a pseudodifferential operator (cf. \cite{Taylor}).

For $0\le \tau \le 1$, the Hilbert spaces $\mathcal{D}(A^{\tau})$ with inner product $\langle \cdot , \cdot \rangle_{\tau}= \langle A^{\tau} \, \cdot , A^{\tau} \, \cdot \rangle$ are the interpolation spaces by quadratic interpolation between $\mathcal{D}(A)$ with inner product $\langle \cdot , \cdot \rangle_{1}= \langle A \, \cdot , A \, \cdot \rangle$ and $\mathcal{H}$ with inner product
$\langle \cdot , \cdot \rangle_{0}= \langle A^{0} \, \cdot , A^{0} \, \cdot \rangle$ (cf. \cite{AAH}). This definition of quadratic interpolation has been used in spectral methods employing periodic Sobolev spaces on the line in \cite{Mercier}. But in the finite element literature the ``real method of
interpolation'' is common (cf. \cite{Bramble} \cite{Brenner-Scott}). Via the usual complexification of real Hilbert spaces both of these interpolation methods are known to be equivalent \cite{Lions-RPR}, \cite{Lions-Magenes}. Our parameter $\tau$ is denoted by $1-\theta$ in these references. Though interpolated norms are commonly used in static ($A u =f $) problems and time dependent problems, for eigenvalue problems such methods seem only to have been used for rates of convergence in standard norms (cf. \cite{Beattie-Greenlee-Ober}, \cite{Beattie-Greenlee}). The use of fractional powers is quite natural in spectral methods, and we will comment later on applicability of our theorems via finite element constructions.

Now observe that by the functional calculus for selfadjoint operators, $T$, i.e., the restriction or extension of $T$ as appropriate, is selfadjoint in $\mathcal{D}(A^{\tau})$ for any $\tau$. We will use the symbol $T$ for all of these operators, and note that in all of these spaces $T$ has the same eigenvalues and eigenvectors.

In theory the upper eigenvalues of $T$ (lower eigenvalues of $A$) can be estimated by the Rayleigh-Ritz method in any of these spaces. The Rayleigh quotient for $T$ in $\mathcal{D}(A^{\tau})$ is
\begin{equation}\notag
\frac{\langle A^{\tau} T u, A^{\tau} u\rangle}
{\langle A^{\tau} u, A^{\tau}  u \rangle}
=\frac{\langle A^{\tau-1} u, A^{\tau} u\rangle}
{\langle A^{\tau} u, A^{\tau}  u \rangle} \text{ for } u \neq 0 \text{ in } \mathcal{D}(A^{\tau}).
\end{equation}
In practice only the Rayleigh quotients with $\tau$ an integer or half integer are potentially useful for computation, and for $\tau\le 0$ this requires a useful formula for $A^{-1}$. Our method is to proceed initially in the abstract, and then use duality type arguments to estimate convergence of standard Rayleigh-Ritz eigenvectors in other interpolation norms. This will resolve the convergence questions raised in the nonlinear rotating string example of the previous section.

The three Rayleigh quotients
\begin{equation} \label{right-lehman-harmonic}
\displaystyle{ \frac{ \langle u,  A u \rangle}{\langle A u, A u
\rangle} }
\end{equation}
corresponding to $\tau=1$,
\begin{equation} \label{regular}
\displaystyle{ \frac{ \langle u,  u \rangle}{a(u)}  = \frac{
\langle u, u \rangle}{\langle A^{1/2} u, A^{1/2} u \rangle} }
\end{equation}
corresponding to $\tau =1/2$, and
\begin{equation} \label{left-lehman}
\displaystyle{ \frac{ \langle T u, u \rangle}{\langle u, u
\rangle} }
\end{equation}
corresponding to $\tau=0$, have been studied in the context of matrix eigenvalue problems in \cite{Beattie-ETNA}. Therein the quotient
\eqref{right-lehman-harmonic} is termed Right Definite Lehman with
zero shift or harmonic Ritz, \eqref{regular} regular Rayleigh-Ritz, and \eqref{left-lehman} Left Definite Lehman or dual harmonic Ritz.

The Rayleigh-Ritz method in $\mathcal{D}(A^{\tau})$, which we now denote by $\mathcal{H}_{\tau}$, proceeds by analogy with regular Rayleigh-Ritz. One selects a family of trial vectors $\{w_i\}_{i=1}^n \subset \mathcal{H}_{\tau}$ and forms the generalized matrix eigenvalue problem
\begin{equation} \notag
\big[\langle w_i, w_j\rangle_{\tau} \big]_{i,j=1}^n \, {\bf x} =\Lambda \big[\langle T w_i, w_j \rangle_{\tau} \big]_{i,j=1}^n \, {\bf x}.
\end{equation}
If the resulting matrix eigenvalues are denoted by
\begin{equation} \notag
\Lambda_1^{(n, \tau)} \le
\Lambda_2^{(n, \tau)} \le \ldots
\Lambda_n^{(n, \tau)},
\end{equation}
the first monotonicity principle again guarantees that for each $i=1, 2, \ldots, n$, $\lambda_i\le \Lambda_i^{(n, \tau)}$. Let
$u_1^{(n, \tau)}, u_2^{(n, \tau)}, \ldots, u_n^{(n, \tau)}$ be the corresponding eigenvectors orthonormalized in $\mathcal{H}_{\tau-1/2}$. Thus $\langle u_i^{(n, \tau)}, u_j^{(n, \tau)}\rangle_{\tau-1/2}=\delta_{ij}$ and
\[\langle u_i^{(n, \tau)}, u_j^{(n, \tau)}\rangle_{\tau}=\Lambda_i^{(n, \tau)} \, \delta_{ij}\]
where $\delta_{ij}$ is the Kronecker symbol.

%\section{Basic Estimates} \label{basic-estimates} -- This is the beginning of the
%old section

Now let $P_{n, \tau}$ be the orthogonal projection operator in $\mathcal{H}_{\tau}$ onto $span \{ w_i \}_{i=1}^n$, the span of the first $n$ trial vectors. In general, $P_{n, \tau}$ depends on $\tau$. If $\{ w_i \}_{i=1}^{\infty}$ is complete in $\mathcal{H}_{\tau}$, $P_{n, \tau} \to I$ strongly in $\mathcal{H}_{\tau}$ as $n\to \infty$.

Completeness of $\{w_i\}_{i=1}^{\infty}$ in
$\mathcal{H}_{\tau}$ is also sufficient to guarantee convergence
of $\Lambda_{i}^{(n,\tau)}$ to $\lambda_i$ as $n\to \infty$ for
each eigenvalue $\lambda_i$ of $A$ below $\lambda_{\infty}$
(cf. \cite{Weinberger-SIAM}). Our interest is in
rate of convergence estimates obtainable by spectral or finite
element methods from $\| \left(I-P_{n,\tau}\right) Q\|_{\tau}$ and
$\| \left(I-P_{n,\tau}\right) Q\|_{\tau-1/2}$, where $Q$ is a
spectral projection for $A$ onto the span of finitely many of the
eigenvectors $u_1, u_2, \ldots$. Note that while, in general,
$P_{n, \tau}$ depends on $\tau$, the spectral projection $Q$ is
independent of $\tau$.

\begin{theorem} \label{theo3.1} Let $Q_{\ell}$ be the
spectral projection for $A$ onto $span \{u_1, \ldots, u_{\ell} \}$. If
$\{w_i\}_{i=1}^{\infty}$ is complete in $\mathcal{H}_{\tau}$, then
for large $n$ and each $j=1, \ldots, \ell$, {\small
\[0 \le \Lambda_{j}^{(n,\tau)} - \lambda_j \le
\Lambda_{\ell}^{(n,\tau)} \, \left( \| \left(I-P_{n,\tau}\right)
Q_{\ell}\|_{\tau-1/2}^2+ 2 \lambda_{\ell} \left( \sum_{i=1}^{\ell}
\lambda_i^{-1} \right)  \| \left(I-P_{n,\tau}\right)
Q_{\ell}\|_{\tau}^2 \right)\] }where $\| \cdot \|_{\tau}$ and  $\| \cdot
\|_{\tau-1/2}$ denote the norms of bounded operators in the
corresponding spaces. Furthermore, there exist eigenvectors
$u_{j,n}$ of $A$ in the $\lambda_j$ eigenspace, normalized in
$\mathcal{H}_{\tau-1/2}$, such that \begin{equation} \notag
\|
u_{j,n}-u_{j}^{(n,\tau)}\|_{\tau}^2 \le 4 \lambda_j \, \left(1+
\frac{\lambda_j}{\kappa \delta}\right)^2 \, \|
\left(I-P_{n,\tau}\right) Q\|_{\tau-1/2}^2+ \Lambda_{j}^{(n,\tau)}
- \lambda_j, \text{ and }
\end{equation}
\begin{equation} \notag
\|
u_{j,n}-u_{j}^{(n,\tau)}\|_{\tau-1/2} \le 2 \, \left(1+
\frac{\lambda_{\ell}}{\kappa \delta}\right) \, \|
\left(I-P_{n,\tau}\right) Q\|_{\tau-1/2}
\end{equation}
where $Q$ is the spectral projection onto the $\lambda_j$
eigenspace, \[\delta=dist (\{\lambda_j\}, \sigma(A)-
\{\lambda_j\}),\] $0<\kappa<1$, and $\sigma(A)$ is the spectrum of
$A$.
\end{theorem}

Theorem \ref{theo3.1} is obtained by translating the proof of the corresponding theorem in \cite{Strang-Fix} into our context, with attention to specific constants in their estimates, and more detail on eigenvectors corresponding to multiple eigenvalues. The proof is included as the Appendix. A strength of the method is that compactness of $T$ is not required, and the eigenvalue and eigenvector estimates scale correctly if $A$ is replaced by $(const.) \, A$ (cf. \cite{Strang-Fix}, \cite{Todor}). But a weakness of the method is that in order to estimate $\lambda_k$, one must employ the estimates for $\lambda_1, \lambda_2, \ldots, \lambda_{k-1}$. Though this is consistent with the progression of the Rayleigh-Ritz method, in such problems as the $L-$shaped membrane considered in \cite{Babuska-Osborn}, \cite{Strang-Fix}, the lowest eigenvalues have eigenvectors that are less smooth than those for higher eigenvalues. This gives slower rates of convergence. In \cite{Babuska-Osborn}, \cite{Chatelin} localized projection estimates (though not in fractional norms), i.e., with $Q_{\ell}$ replaced by the spectral projection for $A$ onto a single eigenspace of $A$ are obtained. Their methods use compactness of $T$, and it is an open problem whether the hypothesis of compactness can be relaxed.

Before proceeding to the abstract theorems for spectral implementation of Theorem \ref{theo3.1} we recall a basic interpolation theorem for use in the sequel.
%\section{Implementation} \label{implementation} -- old beginning of section
Let $\mathcal{V}_1$ and $\mathcal{V}_0$
be Hilbert spaces with $\mathcal{V}_1$ continuously contained in
$\mathcal{V}_0$, written $\mathcal{V}_1 \subset_{c} \mathcal{V}_0$,
i.e., $\mathcal{V}_1 \subset \mathcal{V}_0$ with
the injection from $\mathcal{V}_1$ into $\mathcal{V}_0$
continuous, and $\mathcal{V}_1$ dense in $\mathcal{V}_0$. Then
there is a positive definite selfadjoint operator $S$ in
$\mathcal{V}_0$ such that the inner product $\langle \cdot , \cdot
\rangle_1$ in $\mathcal{V}_1$ is given by
\[\langle u , v \rangle_1 = \langle S u , S v \rangle_0, \]
for all $u, v \in \mathcal{V}_1=\mathcal{D}(S)$, where $\langle
\cdot , \cdot \rangle_0$ is the inner product in $\mathcal{V}_0$.
For $0\le \tau \le 1$, the $\tau^{th}$ interpolation space by
quadratic interpolation between $\mathcal{V}_1$ and
$\mathcal{V}_0$ is the Hilbert space $\mathcal{V}_{\tau}=
\mathcal{D}(S^{\tau})$ with inner product $\langle u , v
\rangle_{\tau} = \langle S^{\tau} u , S^{\tau} v \rangle_0$.

\begin{proposition} \label{prop4.1} Let $\mathcal{W}_1$ and
$\mathcal{W}_0$ be a second pair of Hilbert spaces with
$\mathcal{W}_1 \subset_{c} \mathcal{W}_0$  and $\mathcal{W}_1$ dense in $\mathcal{W}_0$. Further, let $C$ be a continuous linear mapping of
$\mathcal{V}_0$ into $\mathcal{W}_0$ with bound $m_0$ which is
also continuous from $\mathcal{V}_1$ into $\mathcal{W}_1$ with
bound $m_1$. Then for each $\tau \in (0,1)$, $C$ is continuous
from $\mathcal{V}_{\tau}$ into $\mathcal{W}_{\tau}$ with bound
$m_{\tau} \le m_1^{\tau} m_0^{1-\tau}$.
\end{proposition}

Proofs of Proposition \ref{prop4.1} can be found in \cite{AAH},
\cite{Lions-Magenes}. An immediate consequence of this proposition is that changes between equivalent norms of each of $\mathcal{V}_{1}$ and $\mathcal{V}_{0}$ only affects $\mathcal{V}_{\tau}$ up to an equivalent norm. This follows by employing the identity
map for both $C$ and $C^{-1}$.

While rate of convergence results for the linear eigenvalue problem embedded
in the rotating string problem can be obtained by integration by
parts (cf. \cite{Bernardi-Maday}, \cite{Birk-Fix}, \cite {Boyd}), the fractional norm estimates needed for sharp approximation of the bifurcation coefficients require a more extensive theory. It is essential to be able to determine when a function is in the domain of a fractional power of a differential operator from knowledge of the smoothness and the homogeneous boundary conditions satisfied by the function. We now sketch results which resolve this for our purposes.

If $\Omega$ is a bounded domain in $\rz^n$ with smooth boundary $\partial \Omega$, results for general normal homogeneous boundary conditions are obtained in \cite{Grisvard}. Denote by
$H_{\mathbb{B}}^m (\Omega)$ the closed subspace of $H^m(\Omega)$ consisting of the elements $u \in H^m(\Omega)$ for which
\begin{equation} \notag
B_j u = 0 \text{ on } \partial \Omega, 0 \le j \le k,
\end{equation}
where the $B_j$'s are differential operators of order $m_j <m$. The boundary operators $\{B_j\}$ are assumed to have smooth coefficients, to have distinct orders, and to be ``normal'', i.e., to have the form
\begin{equation} \notag
B_j  = \frac{\partial^{m_j} }{\partial{\nu}^{m_j}} + \text{ terms of order at most  } m_j
\end{equation}
on $\partial \Omega$. $\partial/\partial \nu$ denotes the normal derivative. Then
if $\mathcal{V}_1 = H_{\mathbb{B}}^m (\Omega)$, $\mathcal{V}_0 = L^2 (\Omega)$ and
$\tau \in (0,1)$ is such that $m \tau -1/2 \neq m_j$ for any $j=1, \ldots, k$
{\small \begin{equation} \notag
\mathcal{V}_{\tau} = H_{\mathbb{B}}^{m \tau} (\Omega) = \big \{ u \in
H^{m \tau} (\Omega): B_j u=0 \text{ on } \partial \Omega \text{ for all } j
\text{ such that } m_j<m \tau -1/2 \big \}.
\end{equation}}

\noindent Note that for these values of $\tau$ the spaces $H_{\mathbb{B}}^{m \tau} (\Omega)$ are closed subspaces of $H^{m \tau} (\Omega)$. For each of the ``exceptional'' values $\tau$, $\mathcal{V}_{\tau}$ has a strictly stronger norm than that of $H_{\mathbb{B}}^{m \tau} (\Omega)$, the closure of $H_{\mathbb{B}}^{m \tau} (\Omega)$ in $H^{m \tau} (\Omega)$.

\textbf{Remark.} For domains with corners, let $\Omega$ be a Lipschitzian Graph Domain in $\rz^n$ (cf. \cite{AAS2}). When $n=1$, this means an open interval. Then for any positive integer $m$, if we let
$\mathcal{V}_1$ be the Sobolev space $H^m(\Omega)$ and $\mathcal{V}_0=H^0(\Omega)=L^2(\Omega)$, then $\mathcal{V}_{\tau}=H^{m \tau} (\Omega)$, $0\le \tau \le 1$ (cf. \cite{AAS2} where the theory is given in the language of Bessel potentials). For homogeneous Dirichlet boundary conditions and $\Omega$ a Lipschitzian Graph Domain, the following results are contained in \cite{Greenlee}
(cf. also \cite{Bramble}). If, as usual, for $\alpha>0$, $H_0^{\alpha}(\Omega)$ denotes the closure of $C_0^{\infty}(\Omega)$ in $H^{\alpha}(\Omega)$,
$\mathcal{V}_{1}=H_0^{m} (\Omega)$, and $\mathcal{V}_{0}=H^{0} (\Omega)$,
$\mathcal{V}_{\tau}=H_0^{m \tau} (\Omega)$ whenever $\tau \in (0,1)$ is such that $m \tau + 1/2$ is not an integer. For each ``exceptional'' value of $\tau$ the space $\mathcal{V}_{\tau}$ has a strictly stronger norm than that of
$H_0^{m \tau} (\Omega)$. If, in addition, $\partial \Omega$, the boundary of $\Omega$, is smooth enough that the selfadjoint elliptic operator $A$ over $\Omega$ defined by
\begin{equation}
\notag
\left(A u, v\right)_{L^2(\Omega)} = \left(u, v \right)_{H^{m}(\Omega)} \text{ for all } v \in H_0^m(\Omega)
\end{equation}
has domain $H^{2m}(\Omega) \cap H_0^m(\Omega)$ (regularity at the boundary of weak solutions), then if $\mathcal{V}_1=H^{2m}(\Omega) \cap H_0^m(\Omega)$ and
$\mathcal{V}_0=H_0^m(\Omega)$, $\mathcal{V}_{\tau}=H^{2 m\tau} (\Omega) \cap H_0^m(\Omega)$ for all $\tau \in (0,1)$. Since $\mathcal{D}(A^{1/2})=H_0^m(\Omega)$ (cf. \cite{Kato}) this resolves the interpolation problem between $H^{2m}(\Omega) \cap H_0^m(\Omega)$ and $L^2(\Omega)$. This statement is slightly stronger than that of Theorem 5.3 of \cite{Greenlee}, but is exactly what is proved there. As stated here, an additional possible application is to the homogeneous Dirichlet problem for the biharmonic operator on a rectangle in $\rz^2$ (cf. \cite{Grisvard2}).

\section{Spectral Implementation} \label{spectral}

Let $\mathcal{G}_0$ be a Hilbert space consisting of the same elements as $\mathcal{H}_0$ with an equivalent norm. We write $\left( \cdot, \cdot\right)_0$ and
$| \cdot |_0$ for the inner product and norm in $\mathcal{G}_0$, respectively. Further, let $M$ be a positive definite selfadjoint operator in $\mathcal{G}_0$ with
$M^{-1}$ compact. Denote by
\[0<\mu_1\le \mu_2 \le \ldots \le \mu_k \le \ldots \to \infty\]
the eigenvalues of $M$, counted according to multiplicity, with corresponding eigenvectors $q_1, q_2, \ldots, q_k, \ldots$ orthonormal in $\mathcal{G}_0$.  The eigenvectors $\{q_k \}$ will serve as trial vectors for the eigenpairs of $A$. For this purpose let $R_n$ be the $\mathcal{G}_0$-orthogonal
projection onto $span \{q_1, \ldots, q_n \}$, and
let $P_{n, \tau}$ be, as previously, the $H_{\tau}$-orthogonal projection onto
$span \{q_1, \ldots, q_n \}$.
In the rotating string example begun in Section 2, $\mathcal{G}_0$ will be $L^2(a,b)$ with Lebesgue measure, while $\mathcal{H}_0$ will be the $L^2$ space on (a,b) with measure $\rho(x) dx$.

\begin{proposition} \label{prop4.2} Assume that
\begin{equation} \label{4.1}
\mathcal{D}\left(M^{1/2}\right) \subset_{c} \mathcal{D}\left(A^{1/2}\right)
\end{equation}
and let $0 \le \tau \le 1/2$. Then if $\sigma \ge \tau$ and the eigenvectors $u_1, u_2, \ldots, u_{\ell}$ of $A$ are all in $\mathcal{D}\left(M^{\sigma} \right)$,
\[\Lambda_{j}^{(n, \tau)} = \lambda_j + o \left(\mu_{n+1}^{2 \tau - 2 \sigma}\right),  \quad 1 \le j \le \ell, \]
and
\[\|u_{j,n} - u_j^{(n,\tau)}\|_{\tau}= o\left(\mu_{n+1}^{\tau-
\sigma}\right), \quad 1\le j \le \ell,\]
both as $n\to \infty$.
\end{proposition}
Here ``o'' is the usual Landau symbol ``little oh''.

\begin{proof}
First note that \eqref{4.1} implies completeness of $\{q_k\}$ in $\mathcal{D}(A^{1/2})$. Then if $v\in\mathcal{D}(M^{\sigma})$

\begin{align}
\|\left(P_{n, \tau}- I\right) v\|_{\tau} & = \| A^{\tau}
\left(P_{n, \tau}- I\right) v\|_{0} \notag \\
& = \|  A^{\tau} \left(P_{n, \tau}- I\right) \left(R_n-I\right)
v\|_{0} \notag \\
& \le  \|  A^{\tau} \left(R_n- I\right) v\|_{0} \notag
\end{align}
since $P_{n, \tau}$ is an orthogonal projection in $\mathcal{H}_{\tau}= \mathcal{D}(A^{\tau})$. Now since the identity map, I, is continuous from $\mathcal{G}_0$ into $\mathcal{H}_0$ and also from $\mathcal{D}(M^{1/2})$
into $\mathcal{D}(A^{1/2})$, it follows from Proposition \ref{prop4.1}
that for $0\le \tau \le 1/2$,
\[\|A^{\tau} v\|_0 \le c \, |M^{\tau} v|_0, \text{ for all } v \in \mathcal{D}(M^{\tau}).\]
Herein, and in the sequel, $c$ is a generic constant, whose value varies from place to place, but is always independent of the order of the approximation ($n$ here).
Thus for $v \in \mathcal{D}(M^{\sigma})$

\begin{align}
\|\left(P_{n, \tau}- I\right) v\|_{\tau} & \le  c \,
|M^{\tau} (R_n-I) v|_0 \notag \\
&=c \, |(R_n -I)M^{\tau} v|_0,\notag
\end{align}
since $R_n$ is a spectral projection for $M$, and hence $M^{\tau}$. So since $\sigma\ge \tau$ and $v \in \mathcal{D}(M^{\sigma})$,
\begin{align}
\|\left(P_{n, \tau}- I\right) v\|_{\tau}^2 & \le  c
\, \sum_{k=n+1}^{\infty} |\left( M^{\tau} v, q_k \right)_0|^2 \notag \\
& =   c \, \sum_{k=n+1}^{\infty} |\left( M^{\sigma}
v, M^{\tau-\sigma} q_k \right)_0|^2 \notag \\
& =   c \, \sum_{k=n+1}^{\infty} \mu_k^{2 \tau -2
\sigma} \, |\left( M^{\sigma} v, q_k \right)_0|^2 \notag
\\
& \le  c \, \mu_{n+1}^{2 \tau -2 \sigma} \,
\sum_{k=n+1}^{\infty} |\left( M^{\sigma} v, q_k \right)_0|^2 \notag \\
& \le c \, \mu_{n+1}^{2 \tau -2 \sigma} \, |M^{\sigma} v|_0^2 \, \, o(1) \text{ as } n \to
\infty. \notag
\end{align}
In brief,
\[\|(P_{n,\tau} -I) v\|_{\tau}^2 = o\left(\mu_{n+1}^{2\tau-2 \sigma}\right)
\text{ as } n\to \infty.\]
Since the norm in $\mathcal{H}_{\tau}$ is stronger than that in
$\mathcal{H}_{\tau-1/2}$ the proposition follows from Theorem \ref{theo3.1}
\hfill \end{proof}

Note that we have effectively used powers of $M$ to perform ``a continuous abstract integration by parts''. Proposition \ref{prop4.2} provides potentially useful information for $\tau=0$ or $\tau=1/2$. But for $0<\tau<1/2$, the Rayleigh quotient is usually not computable since typically only complicated expressions for equivalent norms are known, and that at best.

We now give the analogue of Proposition \ref{prop4.2} for the harmonic Ritz method, i.e., the Rayleigh quotient with $\tau=1$. $\mathcal{G}_0$ and $M$ are as previously defined.

\begin{proposition} \label{prop4.4}
Assume that
\begin{equation} \label{4.4}
\mathcal{D}\left(M\right) \subset_{c} \mathcal{D}\left(A\right)
\end{equation}
and let $0 \le \tau \le 1$. Then if $\sigma \ge \tau$ and the eigenvectors $u_1, u_2, \ldots, u_{\ell}$ of $A$ are all in $\mathcal{D}\left(M^{\sigma} \right)$,
\[\Lambda_{j}^{(n, \tau)} = \lambda_j + o \left(\mu_{n+1}^{2 \tau - 2 \sigma}\right), 1 \le j \le \ell, \]
and
\[\|u_{j,n} - u_j^{(n,\tau)}\|_{\tau}= o\left(\mu_{n+1}^{\tau-
\sigma}\right), \quad 1\le j \le \ell,\]
both as $n\to \infty$.
\end{proposition}

The proof directly mimics that of Proposition \ref{prop4.2}. As there, potentially useful information has been obtained for $\tau=0, 1/2$, and 1, and for $\tau=1$, this is applicable to the nonlinear string problem. But in that problem, the critical value of $\tau$ to treat the expansion of the nonlinearity is $\tau=3/4$. The following theorem enables us to apply regular Rayleigh-Ritz, i.e., $\tau=1/2$, with $\tau$ large enough to treat the nonlinearity. This requires stronger hypotheses that are only satisfied when $A$ has compact inverse. An example in another context \cite{Beattie-Greenlee}
applies hypotheses like those of the preceding two propositions to $A$ without compact inverse.

\begin{theorem} \label{theo4.3}
Let $\tau>1/2$, assume that $\mathcal{D}(M^{\tau}) =\mathcal{D}(A^{\tau})$,
with equivalent norms, and further assume that
$\mathcal{D}(M^{2\tau-1/2}) \subset_{c} \mathcal{D}(A^{2\tau-1/2})$. Then if for each $j=1, \ldots, \ell$, both
\begin{equation} \label{new4.2}
\frac{\sum_{k=n+1}^{\infty}  \, \mu_{k}^{4 \tau-1} \, |(u_j, q_k)_0|^2}
{\sum_{k=n+1}^{\infty}  \, \mu_{k}^{2 \tau} \, |(u_j, q_k)_0|^2}=O\left(\mu_{n+1}^{2 \tau-1}\right) \text{ as } n\to \infty,
\end{equation}
and $u_j \in \mathcal{D}(M^{\sigma})$ with $\sigma\ge 2 \tau-1/2$, one has
\begin{equation} \notag
\|u_j-u_j^{(n,1/2)}\|_{\tau}= o\left(\mu_{n+1}^{\tau-\sigma} \right) \text{ as } n\to \infty, j=1, \dots, \ell.\end{equation}
\end{theorem}
%\noindent \textbf{Remarks:} Verification of \eqref{new4.2}, which would be trivial for $\tau \le 1/2$, requires rather precise estimates on the generalized Fourier coefficients $(u_j, q_k)_0$. When these are classical Fourier series, as in our nonlinear string example, integration by parts provides the analysis. Convergence of the numerator in \eqref{new4.2}, and therefore of the denominator, places a restriction on $\tau$ relative to $\sigma$. Writing the numerator terms as
%\begin{equation} \notag
%\mu_{k}^{4 \tau-1} \, |(M^{\sigma} u_j, M^{-\sigma} q_k)_0|^2= \mu_{k}^{4 \tau-1-2 \sigma} \, |(M^{\sigma} u_j, q_k)_0|^2,
%\end{equation}
%convergence will definitely hold if $4 \tau -1 -2 \sigma\le 0$, i.e., $\tau \le \frac{\sigma}{2} + \frac{1}{4}$. This sufficient condition will prove to be essentially sharp in the rotating string example.
\begin{proof}
Let $u_{j}^{(n)}= u_{j}^{(n,1/2)}$ and define $w_{j,n}$ by
\[A^{1-2 \tau} w_{j,n} = u_j - u_{j}^{(n)},\]
i.e.,
\[w_{j,n} = A^{2 \tau-1} \left(u_j - u_{j}^{(n)}\right).\]
Our hypotheses guarantee that $w_{j,n}$ is well defined, and that the following manipulations are valid. Then
\begin{align}
\|u_j-u_j^{(n)}\|_{\tau}^2 &= \langle u_j-u_j^{(n)} , A^{1-2 \tau} \, w_{j,n}\rangle_{\tau} \notag \\
&= \langle A^{\tau} (u_j-u_j^{(n)}) , A^{1-\tau} \, w_{j,n}\rangle_{0} \notag \\
&= \langle A^{1/2} (u_j-u_j^{(n)}) , A^{1/2} \, w_{j,n}\rangle_{0}\notag \\
&=a(u_j-u_j^{(n)}, w_{j,n}).\notag
\end{align}
The Rayleigh-Ritz method is such that if $w_{j,n}=0$ for some $n$, it vanishes for all greater $n$, and there is nothing to prove. Otherwise, by the definition of $w_{j,n}$,
\begin{equation} \notag
\|u_j-u_j^{(n)}\|_{\tau} \le \frac{\|u_j-u_j^{(n)}\|_{1/2}\, \|\, w_{j,n}\|_{1/2}}
{\|A^{1-2 \tau} \, w_{j,n}\|_{\tau}}.
\end{equation}
Noting that
\begin{equation} \notag
\frac{\|w_{j,n}\|_{1/2}}
{\|A^{1-2 \tau} \, w_{j,n}\|_{\tau}}= \frac{\|A^{1-2\tau} w_{j,n}\|_{2\tau-1/2}}
{\|A^{1-2 \tau} \, w_{j,n}\|_{\tau}},
\end{equation}
it remains to prove that the latter ratio is $O(\mu_{n+1}^{\tau-1/2})$, as $n\to \infty$, since by Proposition \ref{prop4.2}, $\|u_j-u_j^{(n)}\|_{1/2}=o(\mu_{n+1}^{1/2-\sigma})$ as $n\to \infty$.
For this purpose observe that
\begin{equation} \notag
\frac{\|A^{1-2\tau} w_{j,n}\|_{2\tau-1/2}}
{\|A^{1-2 \tau} \, w_{j,n}\|_{\tau}}=\frac{\|u_j-u_j^{(n)}\|_{2\tau-1/2}}{
\|u_j-u_j^{(n)}\|_{\tau}} \le c \,
\frac{|M^{2 \tau -1/2} (u_j-u_j^{(n)})|_{0}}{
|M^{\tau}(u_j-u_j^{(n)}|_{0}},
\end{equation}
since $\mathcal{D}(A^{\tau})=\mathcal{D}(M^{\tau})$ and $\mathcal{D}(A^{2\tau-1/2})\subset_{c} \mathcal{D}(M^{2\tau-1/2})$. Thus it suffices to show that
\begin{align}
\sup_{v \in span\{u_j, q_1, \ldots, q_{n}\}} \, \frac{|M^{2 \tau-1/2} \, v|_{0}^2}
{|M^{\tau} v|_{0}^2} &=\sup_{x \in span\{M^{\tau} u_j, q_1, \ldots, q_{n}\}} \, \frac{|M^{\tau-1/2} \, x|_{0}^2}
{|x|_{0}^2}\notag \\
&=O(\mu^{2 \tau-1}) \text{ as } n\to \infty, \notag
\end{align}
where we have used the fact that the $q_k$'s are eigenvectors of $M$, and hence of $M^{\tau}$. Then if $R_n$ is the spectral projection for $M$ onto $span\{q_1, \ldots, q_n \}$,
\begin{align}
\sup_{x \in span\{M^{\tau} u_j, q_1, \ldots, q_{n}\}} & \, \frac{|M^{\tau-1/2} \, x|_{0}^2}
{|x|_{0}^2}\notag \\
&=\sup_{x \in span\{(I-R_n) M^{\tau} u_j, q_1, \ldots, q_{n}\}} \, \frac{|M^{\tau-1/2} \, x|_{0}^2}
{|x|_{0}^2}\notag \\
&= \sup_{x \in span\{(I-R_n) M^{\tau} u_j, q_1, \ldots, q_{n}\}} \, \left( \frac{|M^{\tau-1/2} \, R_n \, x|_{0}^2}
{|x|_{0}^2}
+\frac{|M^{\tau-1/2} \, (I-R_n) \, x|_{0}^2}
{|x|_{0}^2}
\right),
\notag
\end{align}
by the Pythagorean theorem. The latter is dominated by
\[\sup_{x \in span\{(I-R_n) M^{\tau} u_j, q_1, \ldots, q_{n}\}} \, \frac{|M^{\tau-1/2} \, R_n \, x|_{0}^2}
{|x|_{0}^2}
\, + \, \sup_{x \in span\{(I-R_n) M^{\tau} u_j, q_1, \ldots, q_{n}\}} \, \frac{|M^{\tau-1/2} \, (I-R_n) \, x|_{0}^2}
{|x|_{0}^2}
\]
which is equal to
\[\sup_{x \in span\{q_1, \ldots, q_{n}\}} \, \frac{|M^{\tau-1/2} \, R_n \, x|_{0}^2}
{|x|_{0}^2}
\, + \, \frac{|M^{\tau-1/2} \, (I-R_n) \, M^{\tau} u_j|_{0}^2}
{|(I-R_n) M^{\tau} u_j|_{0}^2}
\]
since $R_n$ and $I-R_n$ are complementary spectral projections for $M$, and therefore for the powers of $M$. Since $\tau-1/2>0$ the remaining $\sup$ is, by the variational characterization of eigenvalues of $M^{\tau-1/2}$, exactly
\[\frac{|M^{\tau-1/2} \, q_n|_0^2}{|q_n|^2}=\mu_{n}^{2 \tau-1}.\]
The latter ratio is
\begin{equation} \notag
\frac{\sum_{k=n+1}^{\infty}  \, \mu_{k}^{4 \tau-1} \, |(u_j, q_k)_0|^2}
{\sum_{k=n+1}^{\infty}  \, \mu_{k}^{2 \tau} \, |(u_j, q_k)_0|^2},
\end{equation}
which by hypothesis is $O(\mu_{n+1}^{2 \tau-1})$. Since $\mu_{n+1} \ge \mu_{n}$ and $2 \tau-1>0$, this concludes the proof.
\hfill \end{proof}

Note that use of regular Rayleigh-Ritz, rather than ``$\tau$-Ritz'' with $\tau>1/2$, does not change the ``optimal'' rate of eigenvector convergence in the $\tau$-norm as given in Proposition \ref{prop4.4}. Since, when applicable, dual harmonic Ritz gives faster eigenvalue rates of convergence for a given set of trial vectors we now present the corresponding theorem for this case.

\begin{theorem} \label{newtheo4.4}
Let $\tau>0$, assume that $\mathcal{D}(M^{\tau}) =\mathcal{D}(A^{\tau})$,
with equivalent norms, and further assume that
$\mathcal{D}(M^{2\tau}) \subset_{c} \mathcal{D}(A^{2\tau})$. Then if for each $j=1, \ldots, \ell$, both
\begin{equation} \label{new4.4}
\frac{\sum_{k=n+1}^{\infty}  \, \mu_{k}^{4 \tau} \, |(u_j, q_k)_0|^2}
{\sum_{k=n+1}^{\infty}  \, \mu_k^{2 \tau} \, |(u_j, q_k)_0|^2}=O\left(\mu_{n+1}^{2 \tau}\right) \text{ as } n\to \infty,
\end{equation}
and $u_j \in \mathcal{D}(M^{\sigma})$ with $\sigma\ge \tau$, one has
\[\|u_j-u_j^{(n,0)}\|_{\tau} = o\left(\mu_{n+1}^{\tau-\sigma} \right) \text{ as } n\to \infty, j=1, \dots, \ell.\]
\end{theorem}

The proof of Theorem \ref{newtheo4.4} mimics that of Theorem \ref{theo4.3} starting by defining $w_{j,n}$ as $A^{-2 \tau} \, w_{j,n}= u_j-u_j^{(n,0)}$.
The corresponding theorem using a harmonic Ritz calculation with $\tau>1$ would entail defining $w_{j,n}$ by $A^{2-2 \tau} \, w_{j,n}= u_j-u_j^{(n,1)}$. We will not pursue this.

We now examine convergence of Rayleigh-Ritz eigenvectors in the $\mathcal{H}_0$ norm, weaker than the energy norm. The following is an abstraction of a duality argument from finite element theory as presented in \cite{Brenner-Scott}.

\begin{theorem} \label{newtheo4.5}
Assume that $\mathcal{D}(A)=\mathcal{D}(M)$ with equivalent norms $|M v|_0$ and $\|A v\|_0$. Then if $\sigma \ge 1/2$ and $v \in \mathcal{D}(M^{\sigma})$,
\begin{equation} \notag
\|v- P_{n,1/2} v\|_0 =o(\mu_{n+1}^{-\sigma}) \text{ as } n\to \infty.
\end{equation}
Moreover, the conclusion of Proposition \ref{prop4.2} can be supplemented by
\begin{equation} \notag
\|u_j- u_j^{(n,1/2)}\|_0 =o(\mu_{n+1}^{-\sigma}), 1\le j \le \ell, \text{ as } n\to \infty,
\end{equation}
provided that the first $\ell$ eigenvectors, $u_1, \ldots, u_{\ell}$ are in $\mathcal{D}(M^{\sigma})$.
\end{theorem}

\begin{proof}
For simplicity of notation denote $P_{n, 1/2}$ by $P_n$, and for $v \in \mathcal{D}(M^{\sigma})$ let $w_n$ be the unique solution of \begin{equation} \notag
A \, w_n = v- P_n v.
\end{equation}
If $w_n=0$ there is nothing to prove. Otherwise,
\begin{align}
\|v- P_n v\|_0^2 &=\langle v - P_n v, A \, w_n\rangle_0 \notag \\
&= a(v - P_n v, w_n) \notag \\
&=a(v- P_n v, w_n - P_n w_n). \notag
\end{align}
So by the definition of $w_n$,
\begin{align} \label{new4.5}
\|v- P_n v\|_0 &\le \frac{\|v-P_n v\|_{1/2} \, \|w_n - P_n w_n\|_{1/2}}
{\|A w_n\|_0} \notag \\
&\le c \, \frac{\|v-P_n v\|_{1/2} \, \|w_n - P_n w_n\|_{1/2}}
{|M w_n|_0}.
\end{align}
Since $w_n \in \mathcal{D}(A)=\mathcal{D}(M)$, it follows as in the proof of Proposition \ref{prop4.2} that
\begin{equation} \notag
\|w_n - P_n w_n\|_{1/2}^2 \le c \, \mu_{n+1}^{-1} |M w_n |_0^2,
\end{equation}
and since $v \in \mathcal{D}(M^{\sigma})$ with $\sigma \ge 1/2$, it follows similarly that
\begin{equation} \notag
\|v - P_n v\|_{1/2}^2 \le c \, \mu_{n+1}^{1-2 \sigma} \, |M^{\sigma} v|_0^2 \, o(1) \text{ as } n\to \infty.
\end{equation}
So \eqref{new4.5} yields
\begin{equation} \notag
\|v- P_n v\|_0 \le c\, \mu_{n+1}^{-1} \, \mu_{n+1}^{1- 2 \sigma} \,
|M^{\sigma} v|_0^2 \, o(1) = o(\mu_{n+1}^{-2 \sigma}) \text{ as } n \to
\infty
\end{equation}
and the theorem follows from Theorem \ref{theo3.1}.
\hfill \end{proof}

It is possible to use Proposition \ref{prop4.1} to obtain $\|v - P_{n, 1/2} v\|_{\tau} = o(\mu_{n+1}^{\tau-\sigma})$ with $0<\tau <1/2$ as $n \to \infty$ under the hypotheses of Theorem \ref{newtheo4.5}. But this does not immediately extend to a Rayleigh-Ritz eigenvector estimate in the $\tau-$norm. That remains an open problem. The theorem for harmonic Ritz eigenvector estimation in the energy norm is as follows.

\begin{theorem} \label{newtheo4.6}
Assume that $\mathcal{D}(A^{3/2})=\mathcal{D}(M^{3/2})$ with equivalent norms $|M^{3/2} v|_0$ and $\|A^{3/2} v\|_0$. Then if $\sigma \ge 1$ and $v \in \mathcal{D}(M^{\sigma})$,
\begin{equation} \notag
\|v- P_{n,1} v\|_{1/2}^2 =o(\mu_{n+1}^{1-2\sigma}) \text{ as } n\to \infty.
\end{equation}
Moreover, the conclusion of Proposition \ref{prop4.4} can be supplemented by
\begin{equation} \notag
\|u_j- u_j^{(n,1)}\|_{1/2}^2 =o(\mu_{n+1}^{1-2\sigma}), 1\le j \le \ell, \text{ as } n\to \infty,
\end{equation}
provided that the first $\ell$ eigenvectors, $u_1, \ldots, u_{\ell}$ are in $\mathcal{D}(M^{\sigma})$.
\end{theorem}

The proof of Theorem \ref{newtheo4.6} mimics that of Theorem \ref{newtheo4.5}, provided one starts by defining $w_n$ by $A^{3/2} \, w_n = v- P_{n,1} v$, and so is omitted. It is then possible to use Proposition \ref{prop4.1} to obtain $\|v - P_{n,1} v\|_{\tau} = o(\mu_{n+1}^{\tau-\sigma})$ with $1/2<\tau<1$ as $n \to \infty$, but again this does not lead directly to a corresponding eigenvector estimate.

\section{The Model Nonlinear Problem, Part II} \label{string}

There are at least two ways to formulate the spaces $\mathcal{H}_0$ and $\mathcal{G}_0$, and the operators $A$ and $M$ to implement the theory of Section \ref{spectral} for this problem. It is advantageous for our purposes to write \eqref{5.5} as
\begin{equation} \label{5.11}
A v_0 = -\frac{1}{\rho} \, v_0'' = \nu_0 \, v_0, \quad v_0(a)=v_0(b)=0
\end{equation}
with $\mathcal{D}(A) = H^2(a,b) \cap H_0^1(a,b)$. Then $A$ is selfadjoint in $L^2(a,b)$ with measure $\rho(x) dx$, which we take as $\mathcal{H}_0$. The quadratic forms which are needed for harmonic Ritz and regular Rayleigh-Ritz are then found to be
\[\langle A v, A v\rangle =\displaystyle{ \int_a^b \frac{\left(v''\right)^2}{\rho} dx} \text{ on } \mathcal{D}(A),\]
\[\langle A v, v\rangle =\displaystyle{ \int_a^b \left(v'\right)^2 dx} \text{ on } \mathcal{D}(A^{1/2})=H_0^1(a,b),\]
and
\[\langle v, v\rangle =\displaystyle{ \int_a^b v^2 \, \rho dx} \text{ on } \mathcal{H}_0.\]
As $M$ we take
\begin{equation} \label{5.12}
M \, v = - v'', \quad v(a) =v(b)=0,
\end{equation}
selfadjoint in $\mathcal{G}_0=L^2(a,b)$ with Lebesgue measure, $dx$. Then $\mathcal{D}(A)=\mathcal{D}(M)$ with equivalent norms, and the eigenvectors of $M$ generate the usual Fourier sine series. Note that $A v=0$ at $a$ and $b$ if and only if $v''(a) = v''(b)=0$, and since the same is true for $M$, we have $\mathcal{D}(A^2)=\mathcal{D}(M^2)$ (with equivalent norms.) But if $v\in\mathcal{D}(A^3)$,
\begin{align} \label{new5.3}
A^2 \, v &= -\frac{1}{\rho} \, \left( -\frac{1}{\rho} \, v''\right)''\notag \\
&=\frac{1}{\rho} \, \left(\frac{1}{\rho} \, v^{iv} + 2 \left(\frac{1}{\rho}\right)' \, v''' + \left(\frac{1}{\rho}\right)'' \, v'' \right)
\end{align}
must vanish at $a$ and $b$, while if $v \in \mathcal{D}(M^3)$
\[M^2 \, v = v^{iv}\]
must vanish at $a$ and $b$. Independence of the boundary functionals shows that $\mathcal{D}(A^3)$ and $\mathcal{D}(M^3)$ are different closed spaces of $H^6(a,b)$, unless $\rho'(a)=\rho'(b)=0$. These boundary functionals are unstable in the topology of $H^{9/2}(a,b)$ (cf. \cite{Grisvard}). So quadratic interpolation between $\mathcal{D}(A^3)$ and
$\mathcal{D}(A^2)$, and between $\mathcal{D}(M^3)$ and $\mathcal{D}(M^2)$
(as sketched in Section \ref{quadratic-interpolation}) gives
\[\mathcal{D}(A^{\nu}) = \mathcal{D}(M^{\nu}) \quad \text{for all } \nu <2+\frac{1}{4}= \frac{9}{4},\]
with equivalent norms. So one can apply Theorem \ref{theo4.3} with $2 \tau-1/2<9/4$, i.e., $\tau<11/8$ to our problem.

The other standard way to treat \eqref{5.5} is to let $z_0=\sqrt{\rho} \, v_0$ to get
\[A \, z_0 =-\frac{1}{\sqrt{\rho}} \, \left( \frac{z_0}{\sqrt{\rho}}\right)'' = \nu_0 \, z_0, \quad z_0(a)=z_0(b)=0. \]
The so defined operator $A$ is selfadjoint in $L^2(a,b)$ with the usual
Lebesgue measure. So taking this space as both $\mathcal{H}_0$
and $\mathcal{G}_0$, and $M$ as above yields the framework of Section \ref{spectral}.
But an analysis like that in the previous paragraph yields
only $\mathcal{D}(A^{\nu}) = \mathcal{D}(M^{\nu})$ for all $\nu <1+ \frac{1}{4}=\frac{5}{4}$. This yields slower rates of convergence, though Theorem \ref{newtheo4.4} can be applied with $2 \tau -1/2<5/4$, i.e., $\tau<7/8$.

We now proceed to apply the convergence theory of sections \ref{quadratic-interpolation} and \ref{spectral} to solution of \eqref{5.5} and \eqref{5.6} with the operators $A$ and $M$ chosen according to \eqref{5.11} and \eqref{5.12} respectively. As is well known, the eigenvalues $\mu_k$ (all simple) of $M$ are given by
\[\mu_k=\frac{k^2 \pi^2}{(b-a)^2}, \quad k=1, 2, \ldots,\]
with corresponding eigenfunctions (unnormalized)
\[q_k=\sin \frac{k \pi(x-a)}{b-a}, \quad a<x<b.\]

Since $\mu_k$ is proportional to $k^2$ and $\mathcal{D}(A^{\nu})=
\mathcal{D}(M^{\nu})$ for all $\nu<9/4$, Proposition \ref{prop4.2} shows that the Regular Rayleigh-Ritz convergence rates for any index $j$ are
\[\Lambda_j^{(n, 1/2)} = \lambda_j + o(n^{-\alpha}) \text{ for all } \alpha <7,\]
and
\[\|u_j - u_j^{(n,1/2)}\|_{1/2}= o(n^{- \alpha}) \text{ for all } \alpha <7/2.\]
But the result we need for calculation of the bifurcation coefficients is that of Theorem \ref{theo4.3}.

In order to accomplish this we will need to verify that \eqref{new4.2} holds. This will be carried out via integration by parts, which will also slightly improve and simplify the statements of the estimates of the preceding paragraph. Our abstract ``small oh'' estimates are analogues of a classical theorem on Fourier series coefficients of smoothly periodic functions, while the improvement comes from the corresponding ``big oh'' estimates for functions satisfying Dirichlet conditions (cf. \cite{Stakgold}, p. 130). This use of integration by parts is also employed in \cite{Bernardi-Maday}, \cite{Birk-Fix}, \cite{Boyd}. While the abstract theory of Section \ref{spectral} conveniently identifies the possible ``big oh'' power, it does not directly imply such improvement, as follows from examples in a different context (\cite{Greenlee}, pp. 155-156).

So let $(\lambda, u)$ be an eigenpair for the selfadjoint operator $A$, and observe that from the proof of Proposition \ref{prop4.4} that for $0 \le \tau \le 1$,
\begin{equation} \notag
\|\left(P_{n, \tau} - I\right) u \|_{\tau} \le c \, |\left(R_n -I \right) M^{\tau} u|_0.
\end{equation}
$\left(R_n -I \right) M^{\tau} u$ is just the tail of the Fourier sine series of $M^{\tau} u$,
\begin{equation} \notag
|\left(R_n -I \right) M^{\tau} u|_0=\sum_{k=n+1}^{\infty} |(M^{\tau} u, q_k)_0|^2 = \sum_{k=n+1}^{\infty} \mu_k^{2\tau} \, |(u, q_k)_0|^2.
\end{equation}
This is also the quantity of concern for larger values of $\tau$ for application of Theorem \ref{theo4.3}.

Since normalization of $q_k$ in $\mathcal{G}_0=L^2(a,b)$ is independent of $k$, we can bury it in the constant $c$, and integrate by parts six times to obtain
\begin{align} \label{new5.4}
\int_{a}^{b} u \, & \sin \frac{k \pi (x-a)}{b-a} \,  dx = \notag \\
\frac{(b-a)^5}{k^5 \pi^5} \,  & \left(u^{iv}(b) \cos k \pi - u^{iv}(a)-
\displaystyle{
\frac{(b-a)}{k \pi}\, \int_{a}^{b} u^{vi} \, \sin \frac{k \pi (x-a)}{b-a} \, dx}\right).
\end{align}
The boundary terms in the other five integrations by parts all vanish. To see if this predicts the rate of convergence, first note that since the eigenvector $u$ is in the domain of any power of $A$, \eqref{new5.3} implies that
\begin{equation} \notag
u^{iv} = - 2 \rho \, \left(\frac{1}{\rho}\right)' \, u''' \quad \text{ at } a \text{ and } b.
\end{equation}
Then since $\left( \frac{1}{\rho} u''\right)'= - \lambda \, u'$, it follows that
\begin{equation} \notag
u^{iv} = -2 \lambda \, \rho^2 \left(\frac{1}{\rho}\right)'
\,  u' \quad \text{ at } a \text{ and } b.
\end{equation}
Now, $u'(a)\neq 0 \neq u'(b)$ since $u$ is an eigenfunction and $u(a)=u(b)=0$, by uniqueness for the initial value problem. Thus $u^{iv} (a) \neq 0 \neq u^{iv}(b)$, unless $\rho'(a) =\rho'(b)=0$ in which case one obtains a more rapid rate of convergence. Excluding this special case, we have for $0 \le \tau \le 1$,
\begin{equation} \notag
\|\left(P_{n, \tau} - I \right) u \|_{\tau}^2 \le c \, \sum_{k=n+1}^{\infty} k^{4 \tau -10} \le c \, \int_{n}^{\infty} x^{4 \tau-10} = c \, n^{4 \tau -9}
\end{equation}
since $\mu_k = \frac{k^2 \pi^2}{(b-a)^2}$. So it follows as in the proof of Proposition \ref{prop4.2} that
\begin{equation} \notag
\lambda^{(n,1/2)}= \lambda+ O(n^{-7}),
\end{equation}
and
\begin{equation} \notag
\|u- u^{(n,1/2)}\|_{1/2}= O(n^{-7/2}),
\end{equation}
and from that of Theorem \ref{newtheo4.5} that
\begin{equation} \notag
\|u-u^{(n,1/2)}\|_0 = O(n^{-9/2}),
\end{equation}
all as $n\to \infty$. In addition, \eqref{new5.4} implies that there exist constants $c_1$ and $c_2$, independent of $k$, such that
\begin{equation} \notag
c_1 \, k^{-5} \le |(u, q_k)_0| \le c_2 \, k^{-5}
\end{equation}
for large $k$. So the numerator in \eqref{new4.2} is dominated by a constant times
\begin{equation} \notag
\sum_{k=n+1}^{\infty} k^{8 \tau-2} k^{-10} = \sum_{n+1}^{\infty} k^{8 \tau -12} \le \int_{n}^{\infty} x^{8 \tau-12} dx= \frac{n^{8 \tau-11}}{11-8 \tau}
\end{equation}
provided that $8 \tau-12<-1$, i.e., that $\tau<11/8$. Recalling that we must have $\sigma <9/4$, this is consistent with the restriction $\tau\le \frac{\sigma}{2}+\frac{1}{4}$ in the statement of Theorem \ref{theo4.3}. Similarly, the denominator in \eqref{new4.2} is minorized by
\begin{equation} \notag
\sum_{k=n+1}^{\infty} k^{4 \tau} k^{-10} = \sum_{k=n+1}^{\infty} k^{4 \tau-10}
\ge \int_{n+1}^{\infty} x^{4 \tau-10} dx= \frac{(n+1)^{4 \tau-9}}{9 -4 \tau},
\end{equation}
and so the estimate \eqref{new4.2} is
\begin{equation} \notag O\left(\frac{n^{8 \tau-11}}{(n+1)^{4 \tau-9}}\right)=O\left(\mu_{n+1}^{2 \tau-1}
\right)
\end{equation} as $n \to \infty$, for $1/2 \le \tau <11/8$.

Since we have already verified that $\|u- u^{(n,1/2)}\|_{1/2}=O(n^{-7/2})$, it now follows from the proof of Theorem \ref{theo4.3} that
\begin{align} \label{new5.5}
\Lambda^{(n,1/2)}&= \lambda+O(n^{-7}), \text{ and } \notag \\
\\
\|u-u^{(n,1/2)}\|_{\tau}&= O(n^{2 \tau-9/2}),\notag
\end{align}
for $1/2 \le \tau <11/8$ as $n\to \infty$.

We now examine use of Rayleigh-Ritz approximate eigenvalues and eigenvectors to calculate $\nu_1$ as given in \eqref{5.8}. Therein, $(\nu_0, v_0)$ is an eigenpair of $A u = \lambda u$. We will write $(\Lambda^{(n)}, u^{(n)})$ for the corresponding Rayleigh-Ritz eigenpair $(\Lambda^{(n,1/2)}, u^{(n,1/2)})$. Thus we estimate the rate of convergence of
\begin{equation} \label{5.15}
\Lambda^{(n)} \, \left( \rho(x) \, u^{(n)} \, \left({u^{(n)}}'\right)^{2t}, u^{(n)} \right) \text{ to }  \lambda \,
\left( \rho(x) \, u \, (u')^{2t} , u\right).
\end{equation}
Note that since we already have convergence rate estimates for $\Lambda^{(n)}$ to $\lambda$, it suffices to estimate the difference between
\begin{equation}
\int_a^b \rho(x) \, {u^{(n)}}^2  \, \left({u^{(n)}}'\right)^{2t} dx \text{ and }
\int_a^b \rho(x) \, u^2(x) \, \left(u'(x)\right)^{2t} dx. \notag
\end{equation}
Write the difference between these integrals as
\begin{equation} \label{5.161}
\int_a^b \rho(x) \, {u^{(n)}}^2 \left( \left({u^{(n)}}'\right)^{2t}- \left(u'(x)\right)^{2t} \right) dx+
\int_a^b \rho(x) \, \left(u'(x)\right)^{2t} \, \left({u^{(n)}}^2 - u^2\right) dx.
\end{equation}
Since $u'$ is smooth, the second integral in \eqref{5.161} is dominated by
\begin{equation}
\sup_{x \in (a,b)} \, |\rho(x) \, \left(u'(x)\right)^{2t}| \, \int_a^b |{u^{(n)}}^2 - u^2|dx, \notag
\end{equation}
and by the Cauchy-Schwarz inequality the latter integral is dominated by
\begin{equation}
\left( \int_a^b |u^{(n)}+ u|^2 dx\right)^{1/2}  \,  \left( \int_a^b |u^{(n)}- u|^2 dx\right)^{1/2}. \notag
\end{equation}
So by the preceding eigenvector estimates the second integral in \eqref{5.161} is $O(n^{-9/2})$. Since $u^{(n)}$ is uniformly bounded, the first integral in \eqref{5.161} is dominated by
\begin{equation}
\sup_{x \in (a,b)} \, |\rho(x) \, {u^{(n)}}^2| \, \int_a^b |\left({u^{(n)}}'\right)^{2t}- \left(u'(x)\right)^{2t}|dx, \notag
\end{equation}
and by the Cauchy-Schwarz inequality the latter integral is dominated by
\begin{equation}
\left( \int_a^b |{u^{(n)}}'- u'|^2 dx\right)^{1/2}  \,  \left( \int_a^b |\sum_{j=0}^{2t -1}  \left({u^{(n)}}'\right)^j \,
\left(u'\right)^{2t-j-1}|^2 dx\right)^{1/2}. \notag
\end{equation}

\begin{table} [htbp]
  \begin{center}
   \begin{tabular}{|c|c||c||}
    \hline \hline
$t$ &$\nu_1$ & $\nu_1^{(20)}$\\
    \hline
  & &   \\
1 & $\frac{-4(e-1)\pi^4 (\pi^2+\frac{1}{4})(7+ 4 \pi^2)}{e(1+ 20 \pi^2+ 64 \pi^4)}$  & -18.008997020330582    \\ & &  \\
2 & $\frac{-(e^2-1)\pi^6 (\pi^2+\frac{1}{4})(437+ 824 \pi^2+ 144 \pi^4)}
{64 e^2 (1 + 14 \pi^2 + 49 \pi^4+ 36 \pi^6)}$ &  -75.15014087855786  \\
& & 
\\
3 & $\frac{-16 (e^3-1)\pi^8 (\pi^2+\frac{1}{4})(1709+ 5540 \pi^2+ 3856 
\pi^4+ 320 \pi^6)}
{9 e^3 (729 + 9720 \pi^2 + 39312 \pi^4+ 52480 \pi^6+ 16384 \pi^8)}$ & -571.7347727528597      \\ & & 
\\
    \hline
\end{tabular}
\end{center}
\caption{Exact and calculated values of $\nu_1$ for $t=1, 2,$ and $3$.}
    \label{tabla}
\end{table}

  \begin{figure}[tbp18!]
\begin{center}
    \epsfig{figure=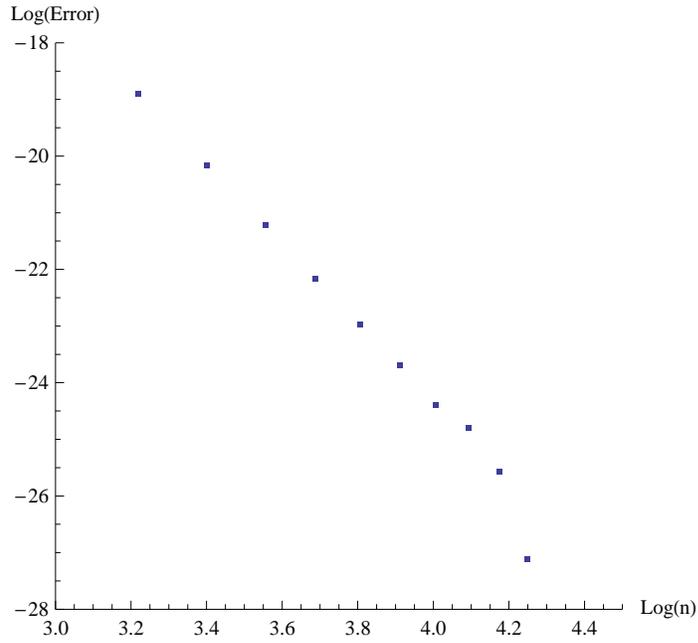,width=0.7 \textwidth}
\caption{Rate of convergence of $\lambda(n)$ to the first eigenvalue $\lambda=\pi^2+1/4$. The equation of the linear fit is given by $y=4.4156-7.22504 x$. Note that $|\lambda-\lambda(20)|=2.795 \times 10^{-8}$.}
\end{center}
     \end{figure}

\begin{figure}[tbp18!]
\begin{center}
    \epsfig{figure=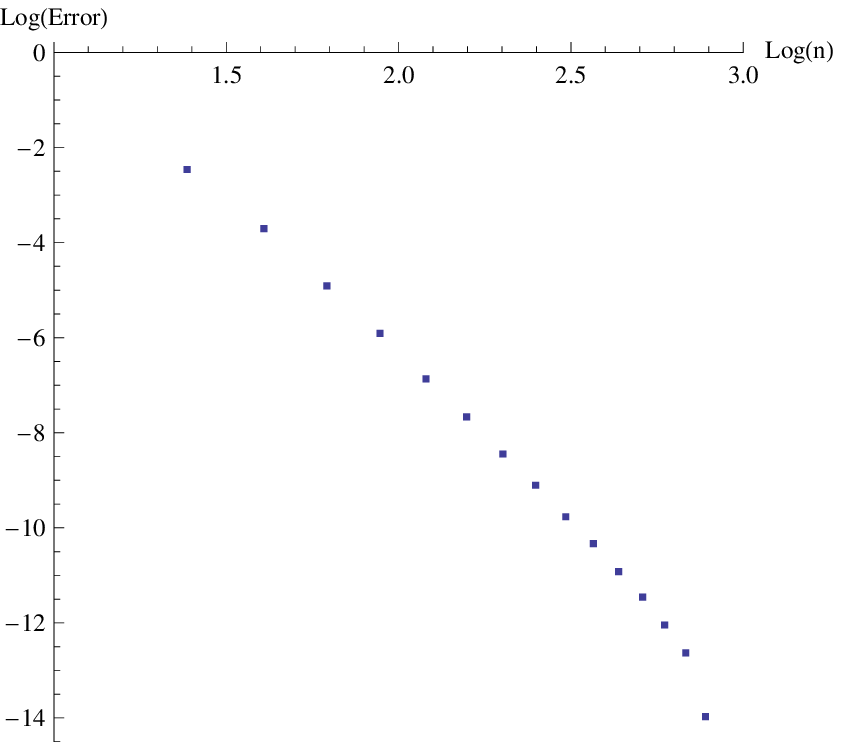,width=0.7 \textwidth}
\caption{Rate of convergence of the bifurcation coefficient $\nu_1^{(n)}$ to $\nu_1$ for $t=1$. The equation of the linear fit is given by $y=8.14125-7.29132 x$. Note that $|\nu_1-\nu_1^{(20)}|=1.86 \times 10^{-6}$.}
\end{center}
    \end{figure}

By the Sobolev inequality \cite{Grisvard2}, \eqref{new5.5} with any
$\tau >3/4$ implies uniform convergence of ${u^{(n)}}'$ to $u'$, so
the latter factor is bounded. Then the eigenvector estimate
\eqref{new5.5} with $\tau=1/2$ gives a rigorous rate of convergence
estimate for the first factor of $O(n^{-7/2})$. In other words, the
rate of convergence estimate for the Rayleigh-Ritz approximation to
the first bifurcation coefficient is the same as that for the
eigenvector in the energy norm, for any $t=1,2, \ldots$. The
following calculations seem to indicate that our convergence rate
estimates for the first bifurcation coefficient are quite
conservative. This may be due to the need to employ the
Cauchy-Schwarz inequality in the absence of $L^1$ estimates (see
Fig. 5.1-5.4). For the problem at hand, with $\rho(x)=1/x^2$, $a=1$, and $b=e$, 
the exact and calculated values of $\nu_1$ for $t=1, 2,$ and $3$ are shown in Table 5.1. The eigenvalue and normalized eigenfunction are $\lambda_1=\pi^2+1/4$ and $v_1(x)=\sqrt{2} \, \sin\left(\ln x \right)$. Calculations were performed using {\em Wolfram Mathematica 6.0}.

\begin{figure}[tbp19!]
\begin{center}
    \epsfig{figure=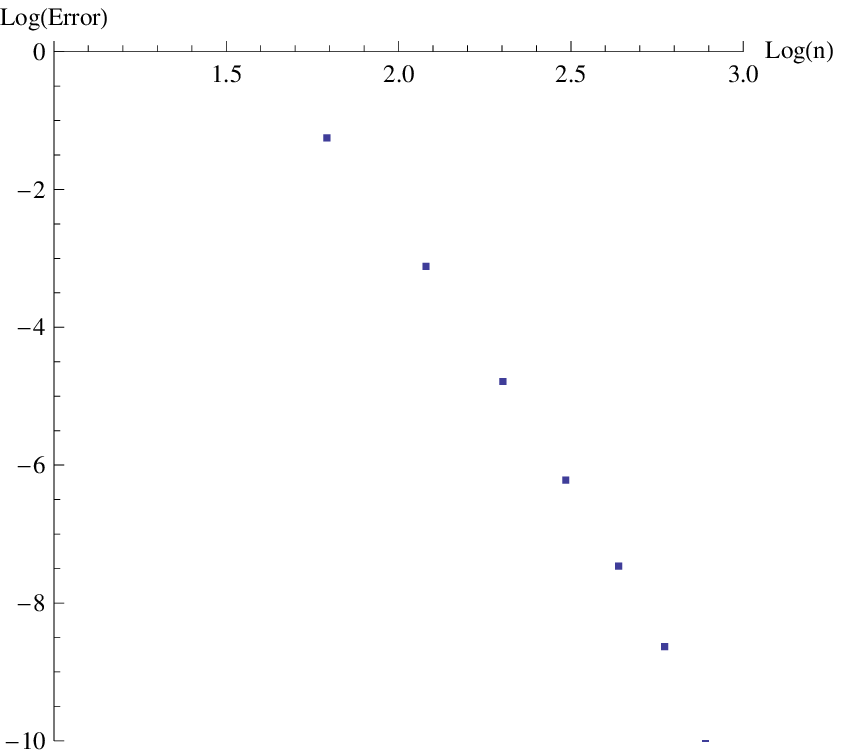,width=0.7 \textwidth}
\caption{Rate of convergence of the bifurcation coefficient $\nu_1^{(n)}$ to $\nu_1$ for $t=2$. The equation of the linear fit is given by $y=10.8453-6.97158 x$. Note that $|\nu_1-\nu_1^{(20)}|=4.3 \times 10^{-5}$.}
\end{center}
    \end{figure}
    
    \begin{figure}[tp19!]
\begin{center}
    \epsfig{figure=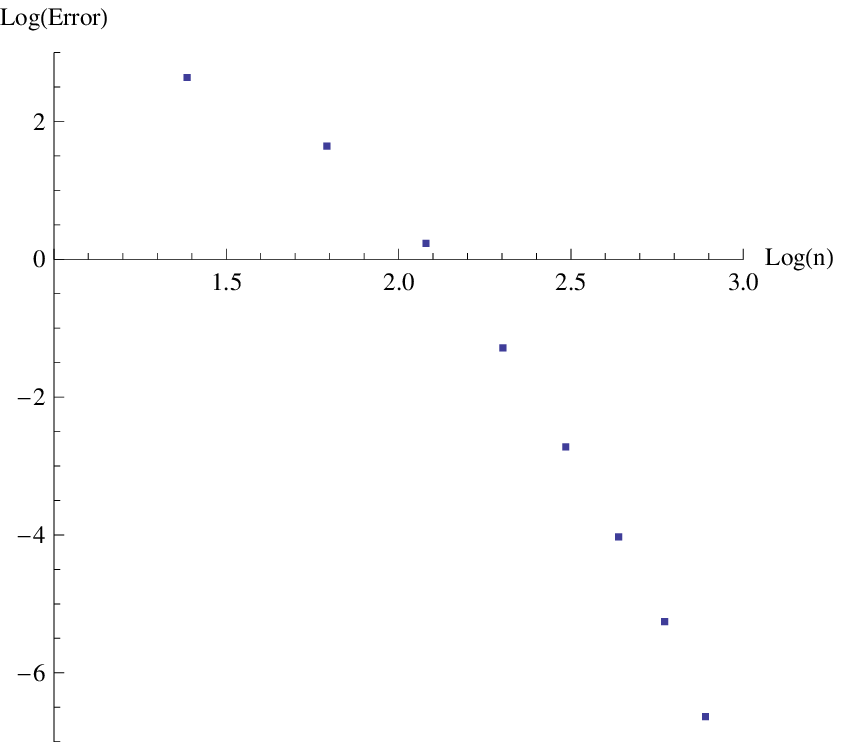,width=0.7 \textwidth}
\caption{Rate of convergence of the bifurcation coefficient $\nu_1^{(n)}$ to $\nu_1$ for $t=3$. The equation of the linear fit is given by $y=12.4435-6.26615 x$. Note that $|\nu_1-\nu_1^{(20)}|=1.1 \times 10^{-3}$.}
\end{center}
     \end{figure}

We now investigate the Rayleigh-Ritz approximation for $v_1$. Convergence of our procedure for $\nu_2, v_2, \ldots$ will then be clear. While this will verify convergence, the method about to be presented appears to be impractical as a numerical algorithm for these higher order terms.

Recall that by \eqref{5.6} $v_1$ is the solution of
\begin{equation} \notag
-v_1''-\nu_0\, \rho(x) v_1=\nu_1 \, \rho(x) \, v_0 +\frac{\nu_0}{2 t} \, \rho(x) \, v_0 \, \left(v_0'\right)^{2 t}, \quad v_1(a)=v_1(b)=0,
\end{equation}
where by \eqref{5.8},
\begin{equation} \notag
\nu_1= -\frac{\nu_0}{2 t} \, \left(\rho(x) v_0 (v_0')^{2t}, v_0\right)=  -\frac{\nu_0}{2 t} \, \langle
v_0 (v_0')^{2t}, v_0 \rangle
\end{equation}
with $\langle v_0, v_0 \rangle=1$. With this value of $\nu_1$ and normalization of $v_0$, this is equivalent to
\begin{equation} \notag
\left(A-\nu_0\right) \, v_1= \nu_1 v_0+ \frac{\nu_0}{2 t} \, v_0 \, \left({v_0}'\right)^{2t},
\end{equation}
or
\begin{equation} \label{new5.8}
v_1= \left(A-\nu_0\right)^{-1} \, \left(\nu_1 v_0+ \frac{\nu_0}{2 t} \, v_0 \, \left({v_0}'\right)^{2t}\right)
\end{equation}
where $A-\nu_0$ has been inverted on the $\mathcal{H}_0$ orthogonal complement of $v_0$, and the component of $v_1$ parallel to $v_0$ is taken to be zero. We now fix attention on the least eigenvalue of $A$, i.e., $\nu_0=\lambda_1$ and $v_0=u_1$, the procedure for higher eigenvalues being entirely analogous. The spectral expansion of \eqref{new5.8} is then
\begin{align} \notag
v_1 &= \frac{\nu_0}{2t} \, \sum_{j=2}^{\infty} \frac{\Big \langle v_0 \, \left(v_0'\right)^{2t}, u_j \Big \rangle_0}{\lambda_j - \lambda_1} \, u_j \notag \\
&= \frac{\lambda_1}{2t} \, \sum_{j=2}^{\infty} \frac{\Big \langle u_1 \, \left(u_1'\right)^{2t}, u_j\Big \rangle_0 }{\lambda_j - \lambda_1}\, u_j.
\end{align}
Recalling the Rayleigh-Ritz normalization
\begin{equation}
\langle u_j^{(n)}, u_k^{(n)}\rangle_{{}_{0}} =\delta_{jk}, \quad   a( u_j^{(n)}, u_k^{(n)}) =\Lambda_j^{(n)} \, \delta_{jk}, \notag
\end{equation}
where we again write $u_j^{(n)}$ for $u_j^{(n,1/2)}$, etc., gives
\begin{equation} \notag
v_1^{(n)}= \frac{\Lambda_1^{(n)}}{2t} \, \sum_{j=2}^{n} \frac{\Big \langle u_1^{(n)} \, \left({u_1^{(n)}}'\right)^{2t}, {u_j^{(n)}} \Big \rangle_0}{\Lambda_j^{(n)} - \Lambda_1^{(n)}} \, u_j^{(n)}.
\end{equation}
As noted previously, by Theorem \ref{theo4.3}, $u_j^{(n)} \to u_j$ in $C_0^1[a,b]$ for each $j$. Thus we have term by term convergence in $\mathcal{H}_{1/2}$, and convergence of $v_1^{(n)}$ to $v_1$ in $\mathcal{H}_{1/2}$ will follow from the dominated convergence theorem for vector valued functions (cf. \cite{Dunford-Schwartz}) provided that $\|v_1^{(n)}\|_{1/2}$ is uniformly bounded with respect to $n$.
Again using the Rayleigh-Ritz normalization we have
\begin{equation}
\|v_1^{(n)}\|_{1/2}^2 = a(v_1^{(n)}, v_1^{(n)}) =
\left(\frac{\Lambda_1^{(n)}}{2t}\right)^2 \, \displaystyle{\sum_{j=2}^{n} \frac{\Big|\Big\langle u_1^{(n)} \, \left({u_1^{(n)}}'\right)^{2t}, u_j^{(n)} \Big\rangle_0\Big|^2 \, \Lambda_j^{(n)}}{\left(\Lambda_j^{(n)} - \Lambda_1^{(n)}\right)^2}}. \notag
\end{equation}
This implies that
\begin{align}
\|v_1^{(n)}\|_{1/2}^2 & \le c \, \|u_1^{(n)} \, \left({u_1^{(n)}}'\right)^{2t}\|_{0}^2 \notag \\
& \le c \, \|u_1 \, \left({u_1}'\right)^{2t}\|_{0}^2,
\notag \end{align}
and convergence in $\mathcal{H}_{1/2}$ is proved. Rates of convergence in various interpolation norms for a fixed finite number of terms in $v_1^{(n)}$ to those in $v_1$ follow from the preceding theory. But this is not sufficient for estimates of the rates of convergence of $v_1^{(n)}$ to $v_1$. Theoretical convergence of $\nu_2^{(n)}$ to $\nu_2$, $v_2^{(n)}$ to $v_2$, etc. follow similarly.

\section{Examples and Remarks}
Let $\mathcal{D}$ denote the unit ball on $\rz^d$ and consider the branching problem
\begin{equation} \notag
-\Delta y + y^{p}  = \lambda \, \rho(x) \, y \quad \text{ in } \mathcal{D}, \quad y =0 \text{ on } \partial \mathcal{D},
\end{equation}
where $\Delta$ is the Laplacian, $p$ is an integer greater than or equal to 2, and the density function $\rho(x)$ is smooth and strictly positive on $\overline{\mathcal{D}}$. As usual $\partial \mathcal{D}$ denotes the boundary of $\mathcal{D}$, $|x|=1$, and
$\overline{\mathcal{D}}$ is the closure of $\mathcal{D}$, $|x|\le 1$. By setting $y=\epsilon^{\frac{1}{p-1}} \, u$, with $\epsilon$ small and positive, we seek normalized solutions to
\begin{equation} \label{new6.1}
-\Delta u + \epsilon \, u^{p}  = \lambda \, \rho(x) \, u \quad \text{ in } \mathcal{D}, \quad  u =0 \text{ on } \partial \mathcal{D}.
\end{equation}
A variant of \eqref{new6.1} for a bounded simply connected domain $\Omega$ in $\rz^2$ is obtained for
\begin{equation} \notag
-\Delta u + \epsilon \, u^{p}  = \lambda \, u \quad \text{ in } \Omega, \quad u =0 \text{ on } \partial \Omega,
\end{equation}
by use of a conformal map $w=f(z)$, $U(z)=u(w)$, to the unit disk $\mathcal{D}$ obtaining
\begin{equation} \label{new6.2}
-\Delta U + \epsilon \, \Big|\frac{dw}{dz}\Big|^2 \, U^{p}  = \lambda \, \Big|\frac{dw}{dz}\Big|^2 \, U \quad \text{ in } \mathcal{D}, \quad U =0 \text{ on } \partial \mathcal{D}
\end{equation}
(cf. \cite{Kuttler-Sigillito}). Provided that $\partial \Omega$ is smooth so that $f$ is conformal on $\partial \Omega$, $\rho(x)=|\frac{dw}{dz}|^2$ has a positive lower bound and our theory applies to \eqref{new6.2} as well as to \eqref{new6.1} (with the exception of strictly radial solutions of \eqref{new6.1}).

Existence of solutions of \eqref{new6.1} or \eqref{new6.2} bifurcating from a simple eigenvalue of the linearized ($\epsilon=0$) problem follows as in \cite{Sattinger}
by use of the Banach space $C^{2, \alpha}(\overline{\mathcal{D}})$ of $C^2$ functions having H\"older continuous derivatives with exponent $\alpha$, $0<\alpha<1$. The formalism sketched in the rotating string problem leads to the bifurcation coefficient for \eqref{new6.1}
\begin{equation} \label{new6.3}
\nu_1=\displaystyle{\int_{\mathcal{D}} \, u^{p+1} \, dx}
\end{equation}
where $u$ is the eigenvector corresponding to the simple eigenvalue $\lambda$ of \eqref{new6.1} with $\epsilon$ set equal to zero and normalized by $\int_{\mathcal{D}} \, \rho(x) \, u^{2} \, dx=1$. An analogous formula holds for \eqref{new6.2}, and we assume that $\rho(x)$ is such that $u$ is unknown so that we must approximate $u$ appropriately to approximate $\nu_1$.

The interpolation theory is directly analogous to that previously detailed in the rotating string problem. The eigenvalue problem for the operator $A$ is given by
\begin{equation} \notag
A \, v_0 = -\frac{1}{\rho} \, \Delta v_0 \quad \text{ in } \mathcal{D}, \quad v_0=0
\text{ on } \partial \mathcal{D},\end{equation}
with $\mathcal{D}(A) = H^2(\mathcal{D}) \cap H_0^1(\mathcal{D})$. The operator $A$ is selfadjoint in $L^2(\mathcal{D})$ with measure $\rho(x) dx$. As $M$ we take
\begin{equation} \notag
M v = - \Delta v
\quad \text{ in } \mathcal{D}, \quad v=0
\text{ on } \partial \mathcal{D},\end{equation}
which is selfadjoint in $\mathcal{G}_0=L^2(\mathcal{D})$ with Lebesgue measure, $dx$. As previously, $\mathcal{D}(A)=\mathcal{D}(M)$ with equivalent norms, and since $A v=0$ on $\partial \mathcal{D}$ if and only if $-\Delta v= M v=0$ on $\partial \mathcal{D}$, we have $\mathcal{D}(A^2)=\mathcal{D}(M^2)$ with equivalent norms (cf. \cite{Lions-Magenes} for the regularity theory). But if $v \in \mathcal{D}(A^3)$
\begin{align}
A^2 \, v &= -\frac{1}{\rho} \, \Delta \left(-\frac{1}{\rho} \, \Delta \right) \notag \\
&= \frac{1}{\rho} \, \left(\frac{1}{\rho} \, \Delta^2 v + 2 \nabla \left(\frac{1}{\rho}\right) \cdot \nabla \left(\Delta v\right) + \Delta \left(\frac{1}{\rho}\right) \Delta v\right), \notag
\end{align}
where $\nabla=$ gradient, must vanish on $\partial \mathcal{D}$. On $\partial \mathcal{D}$, the latter is
\begin{equation} \label{new6.4}
A^2 \, v = \frac{1}{\rho} \, \left(\frac{1}{\rho} \, \Delta^2 v + 2 \nabla \left(\frac{1}{\rho}\right) \cdot \nabla \left(\Delta v\right) \right),
\end{equation}
and if $v \in \mathcal{D}(M^3)$,
\begin{equation} \label{new6.5}
M^2 \, v =\Delta^2 v
\end{equation}
must vanish on $\partial \mathcal{D}$. Since $v \in \mathcal{D}(A)=\mathcal{D}(M)$, $\Delta v= 0$ on $\partial \mathcal{D}$, and so $\nabla \left(\Delta v\right)$ is orthogonal to $\partial \mathcal{D}$, and \eqref{new6.4} and \eqref{new6.5} are the same boundary conditions if $\nabla \left(\frac{1}{\rho}\right)$ is tangential to $\partial \mathcal{D}$ (or vanishes there). We now show that when this is not the case the boundary conditions \eqref{new6.4} and \eqref{new6.5} are not equivalent. This leads to the same critical value $\tau=9/4$ as in the rotating string example.

So suppose there is a point $P$ on $\partial D$ where $\nabla\left(\frac{1}{\rho}\right) \cdot \overrightarrow{n}\neq 0$, where $\overrightarrow{n}$ is the unit outward normal to $\partial \mathcal{D}$. By continuity there is a relative neighborhood $\mathcal{N}$ of $P$ in $\partial \mathcal{D}$ on which $\nabla\left(\frac{1}{\rho}\right) \cdot \overrightarrow{n}$ is non-zero of a fixed sign. Now note that if $u$ is an eigenfunction of $A$ with eigenvalue $\lambda$, the boundary condition $A^2 u=0$ becomes
\begin{equation} \notag
\Delta^2 u = 2 \lambda \rho \nabla\left(\frac{1}{\rho}\right) \cdot \nabla u =
\nabla\left(\frac{1}{\rho}\right) \cdot \overrightarrow{n} \, \frac{\partial u}{\partial n} \quad
\text{ on } \partial \mathcal{D},
\end{equation}
since $\nabla u$ is orthogonal to $\partial \mathcal{D}$. If there is a point of $\mathcal{N}$ at which $\frac{\partial u}{\partial n}\neq 0$ there is a relative neighborhood $\mathcal{N}' \subset \mathcal{N}$ of $P$ on which $\frac{\partial u}{\partial n}$ is non-zero of fixed sign. Then the boundary conditions \eqref{new6.4} and \eqref{new6.5} define different closed subspaces of $H^6(\mathcal{D})$, i.e., $\mathcal{D}(A^3) \neq \mathcal{D}(M^3)$. The other possibility is that $\frac{\partial u}{\partial n}=0$ on $\mathcal{N}$. But then $u$ has zero Cauchy data on $\mathcal{N}$. Since $\mathcal{D}$ is strictly convex at $P$, by \cite{Nirenberg} there is a disk $\mathcal{S}$ centered at $P$ with $u\equiv 0$ on $\mathcal{D}\cap \mathcal{S}$. But then by a theorem of \cite{Muller} (cf. also \cite{Nirenberg}), $u \equiv 0$ on $\mathcal{D}$, contrary to $u$ being an eigenfunction of $A$.

So if $\nabla \left(\frac{1}{\rho}\right)$ is not tangential (or zero) at any point of $\partial \mathcal{D}$ the results of \cite{Grisvard} give, by interpolation between $\mathcal{D}(A^3)$ and $\mathcal{D}(A^2)$, and between $\mathcal{D}(M^3)$ and $\mathcal{D}(M^2)$, that $\mathcal{D}(A^{\nu})=\mathcal{D}(M^{\nu})$ for all $\nu<9/4$. By details in \cite{Grisvard} not elaborated on in Section \ref{rrm}, $\nu$ cannot be increased to $9/4$. Direct application of Proposition 4.1, Proposition 4.2, Theorem 4.5, and Theorem 4.6 now proceed exactly as in the rotating string example. For example, Proposition 4.1 gives the Rayleigh-Ritz eigenvalue approximation
\begin{equation} \notag
\Lambda_j^{(n, 1/2)} = \lambda_j + o(\mu_{n+1}^{1-2 \sigma}) \text{ for all } \sigma <9/4
\end{equation}
as $n \to \infty$. However, in $\rz^d$, $\mu_{n+1} \approx c \, n^{2/d}$ (cf. \cite{Weyl})
so that
\begin{equation} \notag
\Lambda_j^{(n, 1/2)} = \lambda_j + o(n^{2(1-2 \sigma)/d}) \text{ for all } \sigma <9/4
\end{equation}
as $n \to \infty$.

Before looking at rates of convergence for Rayleigh-Ritz or harmonic Ritz approximations to the bifurcation coefficient \eqref{new6.3} we look, as in the rotating string example, to improve these estimates to ``big oh at $9/4$'' by multi-dimensional integration by parts. With $u$ an eigenfunction of $A$, i.e.,
\begin{equation} \notag
-\Delta u = \lambda \, \rho(x) \, u \quad \text{ in } \mathcal{D}, \quad u = 0 \text{ on } \partial \mathcal{D},
\end{equation}
and $q_k$ an eigenfunction of $M$, i.e.,
\begin{equation} \notag
-\Delta q_k = \mu_k \, q_k \quad \text{ in } \mathcal{D}, \quad q_k = 0 \text{ on } \partial \mathcal{D},
\end{equation}
we recall the Lagrange-Green identity
\begin{equation} \label{new6.6}
\int_{\mathcal{D}} \left(u \Delta q_k - q_k \Delta u\right) dx = \int_{\partial \mathcal{D}} \left( u \frac{\partial q_k}{\partial n}-
 q_k \frac{\partial u }{\partial n} \right) dS.
 \end{equation}
Herein, as previously, $\frac{\partial}{\partial n}$ is differentiation in the outward normal direction to $\partial \mathcal{D}$. By first substituting $\Delta^2 u$ for $u$, and then $\Delta u$ for $u$ in \eqref{new6.6}, we obtain
\begin{equation} \label{new6.7}
\int_{\mathcal{D}} u q_k dx = \frac{1}{\mu_k^3} \, \int_{\mathcal{D}} q_k \, \Delta^3 u  dx + \frac{1}{\mu_k^3} \, \int_{\partial \mathcal{D}} \frac{\partial q_k}{\partial n} \, \Delta^2 u dS.
\end{equation}
To obtain \eqref{new6.7} we have also used $u=\Delta u=q_k=0$ on $\partial \mathcal{D}$. \eqref{new6.7} is the multi-dimensional analogue of (5.4), and since
$\frac{\partial q_k}{\partial n}=\nabla q_k \cdot \overrightarrow{n}$ and $\text{div}(\grad q_k)=\Delta q_k$, we expect that
\begin{equation}\notag
\int_{\mathcal{D}} u q_k dx = O(\mu_k^{-5/2}) \text{ as } k \to \infty,
\end{equation}
as in the one dimensional case. But verification of (4.3) is problematic, unless we restrict to the radially symmetric case, i.e., that $\rho$ depends only on the radial variable. The problem arises with the denominator of (4.3).

Now we treat the two dimensional case in some detail, with later comments on higher dimensions. In the usual polar coordinates, the eigenvalue problem for $q_k$ is now
\begin{equation} \notag
-\left(\frac{\partial^2 q_k}{\partial r^2} + \frac{1}{r}\, \frac{\partial q_k}{\partial r} + \frac{1}{r^2}\, \frac{\partial^2 q_k}{\partial \theta^2}\right)=\mu_k \, q_k
\quad \text{ for } r<1, \quad q_k =0 \text{ for } r=1.
\end{equation}
Separation of variables yields the eigenvalues
\begin{equation} \notag
\mu_{m,s}=j_{m,s}^2,
\end{equation}
where $j_{m,s}$ is the $s^{th}$ positive zero of the Bessel function $J_m$. These must be ordered in an increasing sequence to obtain the $\mu_k$'s. For $m=0$, the eigenvalues $j_{0,s}^2$ of $M$ are simple with normalized radial eigenfunctions
\begin{equation} \notag
q_{0,s}(r, \theta)= \frac{1}{\sqrt{\pi}}\, \frac{J_0(j_{0,s} r)}{J_1(j_{0,s})}, \quad s=1, 2, \ldots,
\end{equation}
while for $m=1,2, \ldots$ the eigenvalues $j_{m,s}^2$ of $M$ are double with normalized
eigenfunctions
\begin{equation} \notag
q_{m,s}^{(1)}(r, \theta)= \sqrt{\frac{2}{\pi}} \, \frac{J_m(j_{m,s} r)}{J_{m+1}(j_{m,s})} \, \cos{ m \, \theta}
\end{equation}
and
\begin{equation} \notag
q_{m,s}^{(2)}(r, \theta)= \sqrt{\frac{2}{\pi}} \, \frac{J_m(j_{m,s} r)}{J_{m+1}(j_{m,s})} \, \sin{ m \, \theta}
\end{equation}
for $s=1, 2, \ldots$, (cf. \cite{AS}, \cite{Hilde}). Since the outward normal derivative is given by differentiation with respect to $r$ at $r=1$, we get
\begin{equation} \notag
\frac{\partial q_{0,s}}{\partial r}\big|_{r=1}= \frac{(-1)^{s+1}}{\sqrt{\pi}} \, j_{0,s} \quad , m=0, s=1, 2, \ldots,
\end{equation}
and for $m=1, 2, \ldots, $
\begin{equation} \notag
\frac{\partial q^{(1)}_{m,s}}{\partial r}\big|_{r=1}= \sqrt{\frac{2}{\pi}} \, j_{m,s} \, (-1)^{s+1} \, \cos{m \, \theta}
\end{equation}
and
\begin{equation} \notag
\frac{\partial q^{(2)}_{m,s}}{\partial r}\big|_{r=1}= \sqrt{\frac{2}{\pi}} \, j_{m,s} \, (-1)^{s+1} \, \sin{m \, \theta} \quad
\end{equation}
for $s=1, 2, \ldots$. So the expectation that \eqref{new6.7} gives $\int_{\mathcal{D}} \, u \, q_k dx = O(\mu_k^{-5/2})$ is borne out.

In the radial case we have $\mu_s = j_{0,s} = (s-\frac{1}{4}) \pi + O(\frac{1}{s})$ as $s\to \infty$, \cite{AS} and an analysis like that in the rotating string example shows that (4.3) is satisfied with $\frac{1}{2}\le \tau <\frac{11}{8}$. So in the radial case, using the fact that the Sobolev inequalities (cf. \cite{Grisvard2}) imply that convergence in $\mathcal{H}_{\tau}$ for any $\tau>1/2$ implies uniform convergence, Theorem 4.3 and Theorem 4.5 give convergence of the Rayleigh-Ritz approximation to $\nu_1$ in (6.3) at the rate of $O(n^{-9/4})$. This is the same rate of convergence as that for the Rayleigh-Ritz approximation to the eigenvector $u$ in the $L^2$ norm.

In the non-radial case one can invoke Proposition 4.2 and Theorem 4.6 to obtain the convergence rate $O(\mu_{n+1}^{-7/4})=O(n^{-7/4})$ for the harmonic Ritz approximation to $\nu_1$, since the eigenvector approximation is uniform. The Rayleigh-Ritz method can also be employed via the Sobolev inequality
\begin{equation} \label{new6.8}
H^1(\mathcal{D}) \subset_{c} W^t_{q}(\mathcal{D}) \subset_c
W^0_{q}(\mathcal{D})=L^{q}(\mathcal{D}),
\end{equation}
where $q>2$ and $t=2/q$ (cf. \cite{Grisvard2}). So convergence of $u^{(n)}$ to $u$ in $H^1(\mathcal{D})$ implies that $u^{(n)}$ is bounded in $L^{q}(\mathcal{D})$ for any $q>2$. Thus H\"older's inequality yields convergence of the Rayleigh-Ritz approximation to $\nu_1$ in $L^{q'}(\mathcal{D})$, $\frac{1}{q}+\frac{1}{q'}=1$. This is majorized by convergence in $L^2(\mathcal{D})$ where the rate is at least $O(\mu_{n+1}^{-9/4})= O(n^{-9/4})$ as $n\to \infty$, by Theorem 4.5.

For dimensions higher than two one may employ spherical harmonics and Rayleigh-Ritz as above for $p=2$ and $p=3$ in dimension three, and for $p=2$ in dimension 4. The limitations on $p$ stem from the need for $t\ge 0$ in the analogues of \eqref{new6.8} \cite{Grisvard2}. Harmonic Ritz is directly applicable for any power $p$ in dimension three via uniform convergence. Further conclusions can be obtained from Theorem 4.3 and Theorem 4.5 in the radial case. One can also consider ``derivative nonlinearities'' such as $|\nabla u|^2$ in place of $u^p$ in \eqref{new6.1} for further examples, and the addition of linear potential terms $q(x) u$. The dimensional dependence of the various conclusions comes from the Sobolev inequalities.

We now close with some remarks. For domains with shapes other than balls the main constraint on the preceding methods is the availability of computable approximate eigenvectors, the $q_k$'s. Such are obviously available for rectangles, or for two dimensional domains mapped conformally onto rectangles. But the interpolation results of \cite{Grisvard} are not application. Theorems of \cite{Greenlee} do apply in the case of homogeneous Dirichlet boundary conditions for $A$, but only up to $\tau=1$. So the theory of quadratic interpolation between closed subspaces of Sobolev spaces determined by homogeneous boundary conditions is currently inadequate to obtain results corresponding to the previous over domains with corners.

The finite element method is of course well adapted to the construction of approximate eigenvectors over domains with irregular shapes. One estimates the projections in Theorem 3.1 by use of the ``approximation theorem'', as detailed for $\tau=1/2$ in \cite{Strang-Fix}. Finite element versions of Theorem 3.1, Proposition 4.1, Proposition 4.2, Theorem 4.5, and Theorem 4.6 are attainable. But Theorem 4.3 and Theorem 4.4 each depend on analogue of the ``inverse inequality'' of the spectral Galerkin literature (cf. \cite{Mercier}). It is an open problem as to whether the inverse estimates of finite element theory \cite{Brenner-Scott} can be similarly applied. Quadratic interpolation between closed subspaces of Sobolev spaces determined by homogeneous boundary conditions in the context of the finite element method is considered in \cite{Bramble}, whose results supplement those of \cite{Greenlee}, and in \cite{Brenner-Scott}. In the latter reference the effect of the boundary conditions in quadratic interpolation is not presented, but it does not affect the rates of convergence due to the construction of the finite elements. It does, however, affect the norms of the interpolation spaces in the ``exceptional'' cases.

We have concentrated on the problem of estimation of bifurcation coefficients. But the theorems of Section 4 are also germane to estimating eigenvalue bounds by the so-called ``eigenvector free method'', EVF for short. These bounds are complementary to the Rayleigh-Ritz bounds, and hence provide explicit error estimates. The EVF technique was originated in \cite{Gay-EVF}, and our results herein relate to the convergence theory of
\cite{Beattie-Greenlee-proc}. This problem remains to be explored.

%\newpage

\renewcommand{\theequation}{A.\arabic{equation}}
  % redefine the command that creates the equation no.
  \setcounter{equation}{0}  % reset counter
  \section*{Appendix}  % use *-form to suppress numbering
  \label{appendixA}

We give the proof of Theorem 3.1 in the case $\tau=\frac{1}{2}$,
i.e., regular Rayleigh-Ritz. This enables us to simplify notation
by writing $\| \hskip 0.3 cm \|_{a}$ for $\| \hskip 0.3 cm \|_{1/2}$, $P_n$ for $P_{n,1/2}$, and $\Lambda_j^{(n)}$ for
$\Lambda_j^{(n,1/2)}$. The theorem as stated is obtained simply by substituting $\tau$ for $1/2$ in the appropriate
places. Note that for $\tau <1/2$, this replaces $\mathcal{H}=\mathcal{H}_0$ by $\mathcal{H}=\mathcal{H}_{\tau-1/2}$, a
``negative norm'' space.
\begin{proof}
Let $\mathcal{E}_{\ell} = span \{u_1, \ldots, u_{\ell} \}$ and let $e_{\ell}$ be the set of vectors in $\mathcal{E}_{\ell}$ which are normalized in $\mathcal{H}$. The first step is to prove that if
\begin{equation}
\rho_{\ell}^{(n)} = \max_{u \in e_{\ell}} \big| 2 Re \langle u, u - P_n u \rangle- \| u - P_n u\|^2\big|<1, \notag
\end{equation}
then
\begin{equation}
\Lambda_{\ell}^{(n)} \le \dfrac{\lambda_{\ell}}{1-\rho_{\ell}^{(n)}}. \label{A1}
\end{equation}
\eqref{A1} will follow easily from the minimax principle provided that $P_n \mathcal{E}_{\ell}$ is $\ell$-dimensional.
To see this, observe that if $P_n \mathcal{E}_{\ell}$ has dimension less than $\ell$, there exists $w\in e_{\ell}$ with
$P_n w=0$. But then
\begin{equation} \notag
\rho_{\ell}^{(n)} \ge \big| 2 Re \langle w, w - P_n w \rangle- \| w - P_n w\|^2\big|=\|w\|^2=1.
\end{equation}
Thus, if $\rho_{\ell}^{(n)}<1$, which by completeness of the Rayleigh-Ritz trial vectors is true for $n$ large
\begin{equation} \notag
\Lambda_{\ell}^{(n)} \le \max_{v \in P_n \mathcal{E}_{\ell}} \, \dfrac{a(v)}{\|v\|^2}
=\max_{u \in e_{\ell}} \dfrac{a(P_n u)}{\|P_n u\|^2}.
\end{equation}
Since $P_n$ is an orthogonal projection in $\mathcal{D}(a)$, $a(P_n u) \le a(u)$, and also for $u$ in $e_{\ell}$,
\begin{eqnarray}
\|P_n u \|^2 & = & \|u - (u - P_n u)\|^2 \notag \\
& = & \|u\|^2 -2 Re \langle u, u - P_n u\rangle + \|u - P_n u\|^2 \notag \\
& \ge & 1- \rho_{\ell}^{(n)}. \notag
\end{eqnarray}
Thus if $\rho_{\ell}^{(n)} <1$,
\begin{equation} \notag
\Lambda_{\ell}^{(n)} \le \max_{u \in e_{\ell}} \, \dfrac{a(u)}{1-\rho_{\ell}^{(n)}}
=\dfrac{\lambda_{\ell}}{1-\rho_{\ell}^{(n)}}.
\end{equation}
or equivalently,
\begin{equation} \notag
\dfrac{1}{\Lambda_{\ell}^{(n)}} \ge
\dfrac{1-\rho_{\ell}^{(n)}}{\lambda_{\ell}},
\end{equation}
the corresponding estimate for eigenvalues of $T$.

The second term in the definition of $\rho_{\ell}^{(n)}$ is dominated by
$\|\left(I-P_n\right) \, Q_{\ell}\|^2$, where $Q_{\ell}$ is the orthogonal projection onto $\mathcal{E}_{\ell}$. To obtain an analogous estimate for $\langle u , u- P_n u\rangle$, first write $u \in e_{\ell}$ as $u=\sum_{i=1}^{\ell} c_i u_i$. Then

\begin{eqnarray}
\langle u_i, u -P_n u \rangle & = & \lambda_i^{-1} a(u_i, u - P_n u) \notag \\
& = & \lambda_i^{-1} a(u_i - P_n u_i, u - P_n u), \notag
\end{eqnarray}
since $P_n$ is an orthogonal projection in $\mathcal{D}(a)$. Thus if $u \in e_{\ell}$
\begin{equation}
\langle u, u - P_n u \rangle = \sum_{i=1}^{\ell} c_i \lambda_i^{-1} a(u_i - P_n u_i, u - P_n u), \notag
\end{equation}
and so,
\begin{eqnarray}
2 \big| Re \langle u, u - P_n u \rangle \big|&=& 2 \big| \sum_{i=1}^{\ell} \, Re  \, c_i \lambda_i^{-1} a(u_i - P_n u_i, u - P_n u) \big| \notag \\
& \le & 2 \|u - P_n u\|_{a} \, \sum_{i=1}^{\ell} \|c_i \lambda_i^{-1} \left(I-P_n\right) u_i \|_{a} \notag
\end{eqnarray}
Now, since $u \in e_{\ell}$,
\begin{eqnarray}
\|u - P_n u\|_{a} &=& \|\left(I-P_n\right) Q_{\ell} u \|_{a} \notag \\
&\le&  \|\left(I-P_n\right) Q_{\ell}\|_{a} \, \| u \|_{a}, \notag
\end{eqnarray}
\begin{equation}
\| u \|_{a} = \sum_{i=1}^{\ell} |c_i|^2 \lambda_i \le \lambda_{\ell} \sum_{i=1}^{\ell} |c_i|^2 = \lambda_{\ell}, \notag
\end{equation}
and so,
\begin{equation}
\|u- P_n u\|_{a} \le \lambda_{\ell}^{1/2} \|\left(I - P_n\right) Q_{\ell}\|_{a}. \notag
\end{equation}
Finally,
\begin{eqnarray}
\sum_{i=1}^{\ell} \|c_i \lambda_i^{-1} \left(I-P_n\right) u_i\|_{a}
&\le& \sum_{i=1}^{\ell} |c_i| \lambda_i^{-1} \|\left(I-P_n\right)Q_{\ell} u_i\|_{a}
\notag \\
&\le& \|\left(I-P_n\right)Q_{\ell} \|_{a} \sum_{i=1}^{\ell} |c_i| \lambda_i^{-1} \|u_i\|_{a}  \notag \\
& = & \| \left(I-P_n\right)Q_{\ell} \|_{a}  \sum_{i=1}^{\ell} |c_i| \lambda_i^{-1/2} \notag
\end{eqnarray}
since $\|u_i\|=\lambda_i^{1/2}$. The Cauchy-Schwarz inequality now yields
\begin{equation}
\max_{u \in e_{\ell}} 2 \big| Re \langle u, u - P_n u\rangle\big| \le 2 \lambda_{\ell}^{1/2} \, \sum_{i=1}^{\ell} \lambda_i^{-1/2} \, \| \left(I-P_n\right)Q_{\ell} \|_{a}^2 \notag
\end{equation}
since $\sum_{i=1}^{\ell} |c_i|^2=1$. Thus we have the eigenvalue estimates

\begin{equation} \notag
0 \le \Lambda_{\ell}^{(n)} -\lambda_{\ell} \le \Lambda_{\ell}^{(n)} \left(\|(I-P_n) Q_{\ell}\|^2 + 2 \lambda_{\ell} \left(\sum_{i=1}^{\ell} \lambda_i^{-1}\right)^{1/2} \|(I-P_n) Q_{\ell}\|_{a}^2
\right)
\end{equation}
or equivalently
\begin{equation} \notag
0 \le \frac{1}{\lambda_{\ell}}- \frac{1}{\Lambda_{\ell}^{(n)}} \le \frac{1}{\lambda_{\ell}} \left(\|(I-P_n) Q_{\ell}\|^2 + 2 \lambda_{\ell} \left(\sum_{i=1}^{\ell} \lambda_i^{-1}\right)^{1/2} \|(I-P_n) Q_{\ell}\|_{a}^2
\right).
\end{equation}
To derive the eigenvector estimates, first note that $\|u_{\ell}\|=\|u_{\ell}^{(n)}\|=1$, so
\begin{eqnarray}
a(u_{\ell}-u_{\ell}^{(n)})&=& a(u_{\ell}) -2 Re \, a(u_{\ell},u_{\ell}^{(n)}) + a (u_{\ell}^{(n)}) \notag \\
&=&  \lambda_{\ell} - 2 \lambda_{\ell} Re \, \langle u_{\ell}, u_{\ell}^{(n)} \rangle +\Lambda_{\ell}^{(n)}\notag \\
&=& \lambda_{\ell} \left(2 -2 Re \, \langle u_{\ell}, u_{\ell}^{(n)} \rangle \right) + \Lambda_{\ell}^{(n)}- \lambda_{\ell}\notag \\
&=&   \lambda_{\ell} \left(\|u_{\ell}\|^2-2 Re \, \langle u_{\ell}, u_{\ell}^{(n)} \rangle + \|u_{\ell}^{(n)}\|^2
\right) + \Lambda_{\ell}^{(n)}- \lambda_{\ell} \notag \\
&=&  \lambda_{\ell} \|u_{\ell}- u_{\ell}^{(n)}\|^2+ \Lambda_{\ell}^{(n)}- \lambda_{\ell}.\notag
\end{eqnarray}
So to estimate the error in energy it remains to obtain an estimate for \newline $ \|u_{\ell}- u_{\ell}^{(n)}\|$. To do this we will need the simple identity
\begin{eqnarray}
\left(\Lambda_{j}^{(n)}- \lambda_{\ell}\right) \langle P_n u_{\ell}, u_{\ell}^{(n)}\rangle
&=& a(P_n u_{\ell},u_{j}^{(n)})- \lambda_{\ell} \langle P_n u_{\ell}, u_{j}^{(n)}\rangle \notag \\
&=& \lambda_{\ell} \langle  u_{\ell}- P_n u_{\ell}, u_{\ell}^{(n)}\rangle,
\label{a4}
\end{eqnarray}
valid for any normalized eigenvectors $u_{\ell}$ corresponding to $\lambda_{\ell}$, and $u_{\ell}^{(n)}$ corresponding to
$\Lambda_{\ell}^{(n)}$. For use in the sequel, please observe that in \eqref{a4} the eigenvector $u_{\ell}$ can be selected in an $n$-dependent fashion in the $\lambda_{\ell}$ eigenspace.

So let $\lambda_{\ell}=\lambda_{\ell+1}=\ldots = \lambda_{\ell+K}$ have multiplicity $K+1$, and define $\delta>0$ by
\[\delta=dist(\{\lambda_{\ell}\}, \sigma(A)-\{\lambda_{\ell}\}).\]
Choose $0<\kappa<1$. Then by the preceding there exists $N$, depending on $\kappa$, such that for all $n\ge N$ and $j$ such that $\lambda_j\neq \lambda_{\ell}$,
\[|\Lambda_{j}^{(n)}-\lambda_{\ell}|\ge \kappa \delta.\]

Pick an $\mathcal{H}$-orthonormal basis $\{v_{\ell+k}\}_{k=0}^K$ for the $\lambda_{\ell}$ eigenspace, $Q \, \mathcal{H}$, and write $u\in Q \, \mathcal{H}$ as
$u=\sum_{k=0}^K c_k v_{\ell+k}$. Then since $\{u_{j}^{(n)}\}_{j=1}^{\ell}$ is an orthonormal basis for $P_n \mathcal{H}$
\[P_n u = \sum_{j=1}^n \langle P_n u, u_{j}^{(n)}\rangle u_{j}^{(n)}, u\in Q \, \mathcal{H}.\]
Then by \eqref{a4}
\begin{eqnarray}
\|P_n u - \sum_{\lambda_j=\lambda_{\ell}}^n \langle P_n u, u_{j}^{(n)}\rangle u_{j}^{(n)}\|^2 &=&
\sum_{\lambda_j\neq\lambda_{\ell}} |\langle P_n u, u_{j}^{(n)}\rangle|^2 \notag \\
&=& \sum_{\lambda_j\neq\lambda_{\ell}}^n \left( \frac{\lambda_{\ell}}{\Lambda_j^{(n)}-\lambda_{\ell}}\right)^2 \, |\langle u-P_n u, u_{j}^{(n)}\rangle|^2 \notag \\
&\le & \frac{\lambda_{\ell}^2}{\kappa^2 \delta^2} \,
\sum_{\lambda_j\neq\lambda_{\ell}}^n |\langle u-P_n u, u_{j}^{(n)}\rangle|^2 \notag \\
&\le & \frac{\lambda_{\ell}^2}{\kappa^2 \delta^2} \, \|u-P_n u\|^2. \notag
\end{eqnarray}
Herein $\sum_{\lambda_j=\lambda_{\ell}}^n$ denotes the sum over those indices $j$ from 1 to $n$ for which $\lambda_j=\lambda_{\ell}$, and similarly for the symbol $\sum_{\lambda_j\neq\lambda_{\ell}}^n$.
Thus for any normalized $u \in Q \, \mathcal{H}$
\begin{eqnarray}
\|u- \sum_{\lambda_j=\lambda_{\ell}}^n \langle P_n u, u_{j}^{(n)}\rangle u_{j}^{(n)}\|
&\le& \|u - P_n u\| + \|P_n u - \sum_{\lambda_j=\lambda_{\ell}}^n \langle P_n u, u_{j}^{(n)}\rangle u_{j}^{(n)}\| \notag \\
&\le& \left(1+ \frac{\lambda_{\ell}}{\kappa \delta}\right)\, \|u - P_n u\|. \label{a5}
\end{eqnarray}
This shows that $u \in Q \, \mathcal{H}$ can be approximated by a linear combination of Rayleigh-Ritz eigenvectors.

To proceed further we first assume that $\lambda_{\ell}$ is a simple eigenvalue of $A$ with normalized eigenvector $u=u_{\ell}$. What then needs to be estimated is the difference between $\langle P_n u_{\ell},u_{\ell}^{(n)} \rangle$ and 1 as $n\to \infty$. For this purpose note that we are free to orient the unit vectors
$u_{\ell}$ and $u_{\ell}^{(n)}$ so that $\langle P_n u_{\ell},u_{\ell}^{(n)} \rangle$ is real and nonnegative, which minimizes $\|u_{\ell} - \langle P_n u_{\ell},u_{\ell}^{(n)} \rangle u_{\ell}^{(n)} \|^2$. Note that this may make $u_{\ell}$ depend on $n$ by a scalar of magnitude 1. Then, since $\langle P_n u_{\ell},u_{\ell}^{(n)} \rangle\ge 0$
\begin{align}
\|u_{\ell}\|- \|u_{\ell}- \langle P_n u_{\ell}, u_{\ell}^{(n)}\rangle
u_{\ell}^{(n)}\| &\le \|\langle P_n u_{\ell}, u_{\ell}^{(n)}\rangle
u_{\ell}^{(n)}\| \notag \\
&\le \|u_{\ell}\| + \|u_{\ell}- \langle P_n u_{\ell}, u_{\ell}^{(n)}\rangle
u_{\ell}^{(n)}\| \label{a6}
\end{align}
is the same as
\[|\langle P_n u_{\ell}, u_{\ell}^{(n)}\rangle-1| \le \|u_{\ell}- \langle P_n u_{\ell}, u_{\ell}^{(n)}\rangle
u_{\ell}^{(n)}\|,
\]
so by \eqref{a5}
\begin{eqnarray}
\|u_{\ell}-u_{\ell}^{(n)}\|&\le& \|u_{\ell}- \langle P_n u_{\ell}, u_{\ell}^{(n)}\rangle
u_{\ell}^{(n)}\|+ \|\left(\langle P_n u_{\ell}, u_{\ell}^{(n)}\rangle-1\right) u_{\ell}^{(n)} \| \notag \\
&\le& 2  \|u_{\ell}- \langle P_n u_{\ell}, u_{\ell}^{(n)}\rangle
u_{\ell}^{(n)}\| \notag \\
&\le& 2 \left(1+ \frac{\lambda_{\ell}}{\kappa \delta}\right) \, \|u_{\ell} - P_n
u_{\ell}\|. \label{a7}
\end{eqnarray}
Now again let  $\lambda_{\ell}=\lambda_{\ell+1}=\ldots = \lambda_{\ell+K}$ have multiplicity $K+1$ and let $\Gamma_n$ be the $(K+1) \times (K+1)$ matrix given by
\[\Gamma_n=
\left[\langle P_n v_{\ell+k}, u_{\ell+m}^{(n)} \rangle\right]_{k,m=0}^{K}.\]
Further let $V$ and $U^{(n)}$ be the column vectors
\[V=\left[
\begin{array}{c}
v_{\ell}\\
v_{\ell+1}\\
\vdots \\
v_{\ell+K}
\end{array}\right], \quad
U^{(n)}=\left[
\begin{array}{c}
u_{\ell}^{(n)}\\
u_{\ell+1}^{(n)}\\
\vdots \\
u_{\ell+K}^{(n)}
\end{array}\right].
\]
The estimate \eqref{a5} shows that for each $k=0, 1, \ldots, K$,
\[\sum_{m=0}^K \langle P_n v_{\ell+k}, u_{\ell+m}^{(n)}\rangle u_{\ell+m}^{(n)} \to
v_{\ell+k}\]
in $\mathcal{H}$ as $n\to \infty$, i.e.,
\[\Gamma_n U^{(n)} \to V\]
in $\mathcal{H}^{K+1}$ as $n \to \infty$. It follows that for large $n$, $\Gamma_n$ has rank $K+1$, and so is invertible. Inversion of $\Gamma_n$ would yield a
suitable basis for $Q \, \mathcal{H}$ to complete the proof provided that it is
shown that the inverse is bounded with a bound independent of $n$. It appears
that a diagonalization procedure which enables us to mimic the simple eigenvalue
case yields a more direct route to an explicit estimate. So let
$u=\sum_{k=0}^K c_k v_{\ell+k} \in Q \, \mathcal{H}$ with $\|u\|=1$.
Then
\[P_n u = \sum_{k=0}^K c_k P_n v_{\ell+k} =
\sum_{j=1}^n \langle \sum_{k=0}^K c_k P_n v_{\ell+k}, u_j^{(n)}\rangle
u_j^{(n)}.\]
Since $\Gamma_n$ has rank $K+1$, for each fixed $j=\ell, \ell+1, \ldots,
\ell+K$ the system of $K$ equations in $K+1$ unknowns,
\[\sum_{k=0}^K c_k \langle P_n v_{\ell+k}, u_i^{(n)}\rangle=0, \qquad \ell \le i \neq j \le
\ell+K,\]
has a one parameter family of solutions,
\[\left(c_{j, \ell}^{(n)}, c_{j, \ell+1}^{(n)},
\ldots, c_{j, \ell+K}^{(n)}
\right).
\]
For each $j=\ell, \ell+1, \ldots, \ell+K$, fix the solution by requiring that
$\sum_{k=0}^K |c_{j, \ell+k}^{(n)} |^2 =1$ and that the unit vector
\[u_{j,n} = \sum_{k=0}^K c_{j, \ell+k}^{(n)} v_{\ell+k}\]
in $Q \, \mathcal{H}$ satisfies
\[\langle P_n u_{j,n}, u_j^{(n)}\rangle \ge 0.\]
We now replace the arbitrarily chosen basis $\{v_j\}_{j=\ell}^{\ell+K}$ for $Q \, \mathcal{H}$ with the basis $\{u_{j,n}\}_{j=\ell}^{\ell+K}$. In place of $\Gamma_n$ and $V$ we now have
\[\Delta_n=
\left[\langle P_n u_{\ell+k,n}, u_{\ell+m}^{(n)} \rangle\right]_{k,m=0}^{K}.\]
and
\[U_n=\left[
\begin{array}{c}
u_{\ell,n}\\
u_{\ell+1,n}\\
\vdots \\
u_{\ell+K,n}
\end{array}\right],
\]
respectively. Note that by construction
\[\Delta_n U^{(n)}=\left[
\begin{array}{c}
\langle P_n u_{\ell,n}, u_{\ell}^{(n)} \rangle u_{\ell}^{(n)} \\
\langle P_n u_{\ell+1,n}, u_{\ell+1}^{(n)} \rangle u_{\ell+1}^{(n)} \\
\vdots \\
\langle P_n u_{\ell+K,n}, u_{\ell+K}^{(n)} \rangle u_{\ell+K}^{(n)}
\end{array}\right].
\]
This diagonalization yields the estimate
\[\|u_{\ell+k,n}- u_{\ell+k}^{(n)}\|\le 2 \left(1+ \frac{\lambda_{\ell}}{\kappa \delta} \right) \|u_{\ell+1,n} - P_n u_{\ell+k,n}\|, \quad k=0, 1, \ldots, K,
\]
exactly as in the one dimensional case, \eqref{a6}, \eqref{a7}. Hence, for each $j=1, \ldots, \ell$ there exists a normalized eigenvector $u_{j,n}$ of $A$ such that
\[\|u_{j,n}-u_j^{(n)}\| \le 2 \left(1+ \frac{\lambda_{\ell}}{\kappa \delta} \right)
\|(I-P_n) Q \|\]
where $Q$ is the orthogonal projection onto the $\lambda_j$ eigenspace.
\hfill \end{proof}

%\renewcommand{\theequation}{B.\arabic{equation}}
%  % redefine the command that creates the equation no.
%  \setcounter{equation}{0}  % reset counter
%  \section*{Appendix B: Quadratic Interpolation}  % use *-form to suppress %numbering
%  \label{appendixB}

%\newpage

\end{document}